%% file: ex_article.tex
\documentclass[hidelinks,onefignum,onetabnum]{siamart251216} 
\usepackage{ulem}

\usepackage{amssymb,graphicx,mathrsfs,bm}
\usepackage{enumitem}
\usepackage{microtype}
\usepackage{pgfplots}
\usepackage{placeins}

\input{ex_shared}

\ifpdf
\hypersetup{
  pdftitle={Exact Truncation and Radial Rigidity in Time-Optimal Control},
  pdfauthor={C. Quan, G. Wang, L. Wang and Q. Yan }
}
\fi

\begin{document}

\date{}

\maketitle

\begin{abstract}
We consider minimum-time control for the linear system
\begin{equation*}
\dot{z}(t)=Az(t)+Bu(t),\qquad \|u\|_{L^\infty(0,\infty;\mathbb{R}^m)}\leq 1,
\end{equation*}
with $A\in\mathbb R^{n\times n}$ and $B\in\mathbb R^{n\times m}$.
While the individual point-target and ball-target problems are classical,
we study a different question: when are their optimal controls exactly
compatible, in the sense that, for every nonzero initial state $x$ and all
sufficiently small $\varepsilon>0$, the optimal control for the tolerance ball
$\overline{B}_\varepsilon(0)$ is precisely the restriction of the point-target optimal
control? We prove that this {\it{exact truncation}} property is equivalent
to the rigidity conditions
\begin{equation*}
  B B^\top=\beta I_n,\qquad A+A^\top=2aI_n,
  \qquad \beta>0,\;  a\leq 0,
\end{equation*}
and also to Euclidean radiality of the point-target minimum-time function.
Thus exact truncation holds precisely when the sublevel sets of the
point-target minimum-time function are Euclidean balls centered at the
origin, matching the geometry of the tolerance targets. Moreover, the local
property automatically extends to every $0<\varepsilon<|x|$, and the resulting
optimal times and point-target optimal feedback are both explicit.
\end{abstract}

\medskip


\medskip


\section{Introduction}\mbox{}\par

\medskip

\noindent{\bf The problem.}
Minimum-time control for finite-dimensional linear systems is classical.
The question studied in this paper, however, is not a fixed-target
minimum-time problem. We ask when the optimal controls associated with
different terminal targets are exactly compatible with one another.

We consider the linear control system
\begin{equation}\label{wang1}
    \dot z(t)=Az(t)+Bu(t),\qquad z(0)=x,
\end{equation}
where $A\in\mathbb{R}^{n\times n}$, $B\in\mathbb{R}^{n\times m}$,  $n, m\in \mathbb{N}^+$, and the control $u$
belongs to
\begin{equation*}
    \mathcal U_{ad}:=
    \left\{
        u\in L^\infty(0,\infty;\mathbb{R}^m):
        \|u\|_{L^\infty(0,\infty;\mathbb{R}^m)}\le1
    \right\}.
\end{equation*}
Throughout the paper, for $x\in\mathbb{R}^n$ and $u\in\mathcal U_{ad}$, we write
$z(\cdot;x,u)$ for the solution to \eqref{wang1}. We impose the
spectral dissipativity condition
\begin{equation*}\tag{H1}\label{wang2}
    \operatorname{Re}\lambda\le0,
    \qquad \lambda\in\sigma(A),
\end{equation*}
and the Kalman rank condition
\begin{equation*}\tag{H2}\label{wang3}
    \operatorname{rank}(B,AB,\ldots,A^{n-1}B)=n.
\end{equation*}
For a vector in any finite-dimensional Euclidean space, $|\cdot|$ denotes
the Euclidean norm. $I_n\in \mathbb{R}^{n\times n}$ is the identity matrix. For $\varepsilon\geq 0$, set
$\overline{B}_\varepsilon(0):=\{y\in\mathbb{R}^n:|y|\leq\varepsilon\}$, so that
$\overline{B}_0(0)=\{0\}$. For each $x\in\mathbb{R}^n\setminus\{0\}$ and
$0\leq\varepsilon<|x|$, let $(P_{\varepsilon,x})$ denote the minimum-time problem with
target $\overline{B}_\varepsilon(0)$, and define its optimal time by
\begin{equation*}
T_\varepsilon^*(x)
:=
\inf\left\{
T>0:\ z(T;x,u)\in \overline{B}_\varepsilon(0)
\text{ for some } u\in\mathcal U_{ad}
\right\}.
\end{equation*}
A control $u\in\mathcal U_{ad}$ is admissible for $(P_{\varepsilon,x})$ if
$z(T;x,u)\in \overline{B}_\varepsilon(0)$ for some $T>0$, and it is optimal if
$z(T_\varepsilon^*(x);x,u)\in \overline{B}_\varepsilon(0)$. Thus $(P_{\varepsilon,x})$ is
the ball-target problem when $\varepsilon>0$, and $(P_{0,x})$ is precisely the
point-target problem of steering the state to the origin.

The assumptions (H1)--(H2) are precisely the necessary and
sufficient conditions for $(P_{0,x})$ to admit an optimal control for
every $x\in\mathbb{R}^n\setminus\{0\}$; see, for instance,
\cite{Sontag1998,WangWangXuZhang}. Thus they are the natural standing
assumptions for the family of minimum-time problems considered here. The
basic existence, uniqueness, regularity, Pontryagin maximum principle, and
optimal-time convergence facts used below are collected in
Proposition~\ref{prop:basic-min-time}.

For $0\leq\varepsilon<|x|$, let $u_{\varepsilon,x}^*$ denote the optimal control of
$(P_{\varepsilon,x})$, and set
$z_{\varepsilon,x}^*(t):=z(t;x,u_{\varepsilon,x}^*)$ for
$0\leq t\leq T_\varepsilon^*(x)$. {\it{Here and throughout the paper, we use the piecewise continuous and
left-continuous representative specified in
Proposition~\ref{prop:basic-min-time}}} and extend $u_{\varepsilon,x}^*$ by zero to
$(T_\varepsilon^*(x),\infty)$. We introduce the point-target minimum-time function
$V_0(x):=T_0^*(x)$ for $x\ne0$, with $V_0(0):=0$.

Our central question is whether replacing the point target by a small
tolerance ball changes only the stopping time, while leaving the optimal
control unchanged before that time. More precisely, we say that $(A,B)$
has the {\it{exact truncation property}} if, for every $x\not=0$, there
exists $\varepsilon_x\in(0,|x|)$ such that
\begin{equation}\label{wang4}
u_{\varepsilon,x}^*(t)=u_{0,x}^*(t),
\qquad
0<t\le T_\varepsilon^*(x),
\qquad
0<\varepsilon<\varepsilon_x.
\end{equation}

\medskip
\noindent{\bf Motivation and significance.}
The point-target problem $(P_{0,x})$ is classical. In applications,
however, a terminal state is measured, estimated, or realized only with
finite accuracy. Replacing the exact target $\{0\}$ by a tolerance ball
$\overline{B}_\varepsilon(0)$ is therefore a natural relaxation. A first question is
whether the optimal  controls for $(P_{\varepsilon,x})$ converge to
those of $(P_{0,x})$ as $\varepsilon\downarrow 0$.
 The property \eqref{wang4} asks for substantially
more: the two optimal controls must agree exactly on their common time
interval. Thus the terminal tolerance may change the stopping time, but not
the control applied before that time.

Accordingly, our problem is not a stability question for a single
fixed-target minimum-time problem. It is a cross-target compatibility
question: we ask whether the point-target optimal synthesis also solves the
nearby ball-target problems up to their respective stopping times.
Classical minimum-time theory, by contrast, fixes a terminal target and
studies the corresponding reachable sets, value function, switching
structure, or optimal synthesis.

At first sight, such compatibility may appear almost automatic. In the
scalar case $n=m=1$, exact truncation indeed holds under
(H1)--(H2): replacing the origin by a small tolerance
interval merely stops the same time-optimal process earlier. This
one-dimensional intuition is misleading in higher dimensions. A
trajectory that is optimal for reaching the origin necessarily crosses
$\partial B_\varepsilon(0)$, but another controlled trajectory may enter
$\overline{B}_\varepsilon(0)$ earlier at a different point. Hence point-target
optimality does not in general imply ball-target optimality.

\medskip
\noindent{\bf Main results.}
Our main theorem gives a complete answer.

\begin{theorem}[Characterizations of the exact truncation property]\label{wang5}
    Assume \textnormal{(H1)} and \textnormal{(H2)}. Then the following statements are
equivalent.
\begin{enumerate}
  \item[\textnormal{(i)}] $(A,B)$ has the exact truncation property, namely,
\eqref{wang4} holds.
  \item [\textnormal{(ii)}] The system satisfies the algebraic rigidity condition: there exist
constants $\beta>0$ and $a\leq 0$ such that
\begin{equation}\label{wang6}
BB^\top=\beta I_n,
\qquad
A+A^\top=2aI_n.
\end{equation}
  \item [\textnormal{(iii)}]  The point-target minimum-time function is Euclidean radial, that is,
there exists a function $\Psi:[0,\infty)\to[0,\infty)$ with $\Psi(0)=0$
such that
\begin{equation}\label{wang7}
V_0(x)=\Psi(|x|),
\qquad x\in\mathbb{R}^n.
\end{equation}
\end{enumerate}

Moreover, once any of these equivalent conditions holds, the equality in
\eqref{wang4} holds for every $x\not=0$ and every
$0<\varepsilon<|x|$.
\end{theorem}

\begin{remark}\label{2026-9-6-0}
The identity $BB^\top=\beta I_n$ implies $\textnormal{rank}B=n$. Hence $m\geq n$.
\end{remark}

\begin{corollary}[Explicit optimal times]\label{wang8}
If one of \textnormal{(i)}--\textnormal{(iii)} in
Theorem~\ref{wang5} holds, then, for every $x\not=0$ and every
$0\leq\varepsilon<|x|$,
\begin{equation}\label{wang9}
T_\varepsilon^*(x)
=
\int_\varepsilon^{|x|}
\frac{d\xi}{\sqrt{\beta}-a\xi},
\end{equation}
where
\begin{equation}\label{wang10}
\beta=\frac{1}{n}\operatorname{tr}(BB^\top)>0,
\qquad
a=\frac{1}{n}\operatorname{tr}A\le0.
\end{equation}
\end{corollary}

\begin{corollary}[Optimal feedback and Bellman equation]
\label{wang11}
If one of \textnormal{(i)}--\textnormal{(iii)} in
Theorem~\ref{wang5} holds, then the following statements hold.
\begin{enumerate}

\item[\textnormal{(i)}] The value function $V_0$ belongs to
$C^{0,1}(\mathbb{R}^n)\cap C^\infty(\mathbb{R}^n\setminus\{0\})$ and satisfies the
classical minimum-time Bellman equation
\begin{equation}\label{wang12}
1+\min_{|u|\le1}
\langle\nabla V_0(x),Ax+Bu\rangle=0,
\qquad x\not=0.
\end{equation}

\item[\textnormal{(ii)}] The problem $(P_{0,x})$ admits the optimal state feedback
\begin{equation}\label{wang13}
K(x)=-\frac{B^\top x}{\sqrt{\beta}\,|x|},
\qquad x\ne0,
\end{equation}
where $\beta$ is given by \eqref{wang10}.

\item[\textnormal{(iii)}] For every $x\ne0$, $K(x)$ is the unique minimizer of the pointwise
minimization problem
\begin{equation}\label{wang14}
\min_{|v|\le1}
\langle\nabla V_0(x),Ax+Bv\rangle.
\end{equation}
In particular, for the optimal state $z_{0,x}^*$ of $(P_{0,x})$,
$u_{0,x}^*(t)=K(z_{0,x}^*(t))$ for a.e. $t\in(0,T_0^*(x))$.
\end{enumerate}
\end{corollary}

\begin{remark}[Geometric meaning of Theorem \ref{wang5}]
{\it
Condition \textnormal{(ii)} expresses two isotropy properties. The
identity $BB^\top=\beta I_n$ means that the control acts with the same
strength in every direction of the state space. Geometrically, the image
under $B$ of the unit ball in the control space, which is generally an
ellipsoid, is here the Euclidean ball of radius $\sqrt{\beta}$ in the
state space.

Moreover, $A+A^\top=2aI_n$ means that the symmetric part of $A$ acts
identically in every direction, while its skew-symmetric part generates
only a rotation. Indeed, as shown in the proof, writing $A=aI_n+S$ with
$S^\top=-S$, the optimal trajectory of $(P_{0,x})$ has the form
\(z_{0,x}^*(t)=r_*(t)e^{St}x/|x|.\) 
Thus the optimal trajectory need not be a straight line: its radial
motion toward the origin may be accompanied by the rotation generated by
$S$.

Condition \textnormal{(iii)} says that the sublevel sets
$\{x\in\mathbb{R}^n: V_0(x)\leq T\}$ are Euclidean balls centered at the origin.
Thus exact truncation occurs precisely when these point-target
minimum-time sublevel sets have the same concentric Euclidean-ball geometry
as the tolerance targets $\overline{B}_\varepsilon(0)$.
}
\end{remark}

\begin{remark}[Perturbations of exact truncation]\label{wang127}
{\it
The rigidity condition \eqref{wang6}, and hence exact truncation,
is generally destroyed by small perturbations of $(A,B)$. This raises a
natural stability question. If $(A_\delta,B_\delta)\rightarrow (A,B)$ with
$(A,B)$ satisfying \eqref{wang6}, one may ask whether the failure
of exact truncation remains small, for instance whether
$\|u_{\varepsilon,x}^{\delta,*}-u_{0,x}^{\delta,*}\|_{L^\infty(0,T_{\varepsilon,\delta}^*(x);\mathbb R^m)}\rightarrow 0$
as $(\delta,\varepsilon)\rightarrow (0,0)$. A quantitative theory of this approximate
truncation phenomenon is left for future work.
}
\end{remark}

\medskip
\noindent{\bf Relation to previous work.}
Classical time-optimal control for linear systems has been studied from
several complementary viewpoints, including reachable-set geometry, the
Pontryagin maximum principle, bang--bang structure, optimal synthesis,
and the regularity and geometry of the minimum-time function; see, among
others,
\cite{LaSalle1960, HermesLaSalle1969, Hajek1971, Hermes1972,
Brunovsky1978, Bardi1989, EvansJames1989, Staicu1989, GozziLoreti1999,
ColomboNguyen2013}. The geometric relation between adjoint directions
and supporting points of reachable sets is part of this classical
theory; see, for example, \cite{HermesLaSalle1969, Hermes1972}. The
connection between fixed-time reachable sets and level sets of the
minimum-time function is also classical. Mart\'{\i}nez-Legaz
\cite{MartinezLegaz1987}, in particular, developed a systematic
level-set approach to minimal-time functions for linear control
processes. Our use of fixed-time controllable sets as sublevel sets of
$V_0$ belongs to this framework.

A different classical line of work concerns symmetries and
identification of linear systems through their reachable sets. Chukwu
\cite[Theorem~4]{Chukwu1974} characterized reachable-set symmetries
under the hypothesis that the control constraint has only countably many
extreme points; this hypothesis does not cover the Euclidean ball in
dimensions greater than one. Hautus and Olsder
\cite[Theorem~1]{HautusOlsder1973} studied a related identification
problem in which coincidence of the reachable sets of two linear systems,
under suitable hypotheses, determines the underlying system matrices up
to a restricted ambiguity.

\medskip
\noindent{\bf Novelty and contributions.}
Against this background, the contribution of the present paper has three
distinct aspects.

{\it{Problem-level novelty.}}
The point-target and tolerance-ball minimum-time problems associated with
\eqref{wang1} are classical, as are perturbative questions with respect
to the target (see, for instance, \cite{Bonifacius-Pieper, Nakagiri}). The new object here is the exact identity \eqref{wang4}:
for every nonzero initial state and all sufficiently small $\varepsilon$ the
ball-target optimal control is exactly the truncation of the point-target
one. Thus we study a local cross-target identity, rather than convergence
or stability as $\varepsilon\downarrow0$.

{\it{Result-level novelty.}}
The principal new result is the equivalence between the local exact
truncation property \eqref{wang4}, the algebraic rigidity
\eqref{wang6}, and the Euclidean radiality
\eqref{wang7}. To our knowledge, such a characterization has
not appeared previously. In particular, unlike the symmetry and
identification results above, the Euclidean structure is not imposed or
inferred from a comparison of reachable sets; it is forced by a local
identity between optimal controls for two different terminal targets. A
further new feature is the local-to-global phenomenon: exact truncation
assumed only for sufficiently small $\varepsilon$ automatically extends to
every $0<\varepsilon<|x|$. The explicit optimal-time formula
\eqref{wang9} and the feedback
\eqref{wang13} are additional consequences of the same
rigidity.

{\it{Method-level novelty.}}
The methodological contribution is a cross-target terminal-contact mechanism
for extracting algebraic structure from optimal controls. Lemma~\ref{wang18}
realizes arbitrary terminal adjoint directions for the point-target problem,
while exact truncation transfers the corresponding optimal control to a
shrinking ball target, whose terminal transversality fixes the adjoint
direction by the radial normal. Comparing these two terminal geometries in arbitrary directions yields the
two rigidity identities in \eqref{wang6} from the local cross-target property
\eqref{wang4}.

\section{A realization lemma}\mbox{}\par

\medskip

\noindent We first record the standard facts about the family
$(P_{\varepsilon,x})$ that will be used in the sequel.

\begin{proposition}[Basic facts for the minimum-time problems]
\label{prop:basic-min-time}
Assume \textnormal{(H1)} and \textnormal{(H2)}. Let $x\in\mathbb{R}^n\setminus\{0\}$ and
$0\leq\varepsilon<|x|$.
\begin{enumerate}
\item[\textnormal{(i)}]  The optimal time $T_\varepsilon^*(x)$ is finite and is attained by an
optimal control $u_{\varepsilon,x}^*\in\mathcal U_{ad}$, which is unique up to
almost-everywhere equality. If $\varepsilon>0$, then
$|z_{\varepsilon,x}^*(T_\varepsilon^*(x))|=\varepsilon$.

\item[\textnormal{(ii)}]  If $\varepsilon>0$ and
$\eta_{\varepsilon,x}:=
z_{\varepsilon,x}^*(T_\varepsilon^*(x))/\varepsilon$, 
then $|\eta_{\varepsilon,x}|=1$, and the terminal transversality condition and
the Pontryagin maximum principle give
\begin{equation*}
u_{\varepsilon,x}^*(t)
=
-\frac{B^\top e^{A^\top(T_\varepsilon^*(x)-t)}\eta_{\varepsilon,x}}
{\left|B^\top e^{A^\top(T_\varepsilon^*(x)-t)}\eta_{\varepsilon,x}\right|},\;\;\;\;t\in (0,T_\varepsilon^*(x)),
\end{equation*}
whenever the denominator is nonzero.

\item[\textnormal{(iii)}]  The optimal control has a uniquely determined representative on
$(0,T_\varepsilon^*(x)]$ that is piecewise continuous and left-continuous, with
only finitely many discontinuities. Throughout the paper, we still use
$u_{\varepsilon,x}^*$ to denote this representative.

\item[\textnormal{(iv)}]  For each fixed $x\not=0$, 
$T_\varepsilon^*(x)\longrightarrow T_0^*(x)$
 as $\varepsilon\downarrow 0$.
\end{enumerate}
\end{proposition}

\begin{proof} We give a sketch proof of (i)-(iv) one by one.

Firstly, the existence of optimal controls follows from Corollary 3.2 and Theorem 3.11 in \cite{WangWangXuZhang} (or Theorem 4.6 in Chapter 10 of
\cite{Sontag1998}). The
bang-bang property of optimal control is implied by \cite[Theorem 6.1]{WangWangXuZhang} and then  the uniqueness of optimal control holds.
Since $x\notin \overline{B}_\varepsilon(0)$, continuity
of the optimal state implies
$|z_{\varepsilon,x}^*(T_\varepsilon^*(x))|=\varepsilon$ when $\varepsilon>0$.

Secondly, (ii) follows from Theorem 4.2 and Theorem 4.1 in \cite{WangWangXuZhang}.

Thirdly, for the ball target,  the function $B^\top e^{A^\top(T_\varepsilon^*(x)-\cdot)}\eta_{\varepsilon,x}$  is real analytic and, by (H2), is
not identically zero. Hence it has only finitely many zeros on the compact
optimal-time interval, and $u_{\varepsilon,x}^*(\cdot)$ has a left
limit at each zero. This gives the representative in (iii).
For the point-target problem, the same representative property is also
recorded in \cite[Lemma~5.3]{QinWangYu2021}.

Finally, the strict monotonicity of $\{T_\varepsilon^*(x)\}_{\varepsilon>0}$ and a standard contradiction argument lead to the convergence in (iv).
\end{proof}

Let $T>0$ and $\xi\in\mathbb{R}^n\setminus\{0\}$. We define
\begin{equation}\label{wang15}
q_{\xi,T}(t):=B^\top e^{A^\top(T-t)}\xi,
\qquad 0\le t\le T.
\end{equation}
By \eqref{wang3}, $q_{\xi,T}(\cdot)$ is not identically zero. Since $q_{\xi,T}(\cdot)$ is
analytic, it has only finitely many zeros in $[0,T]$, and
$q_{\xi,T}(s)/|q_{\xi,T}(s)|$ admits a left limit at each zero
$t\in(0,T]$. Define
\begin{equation}\label{wang16}
u_{\xi,T}(t):=
\begin{cases}
\dfrac{q_{\xi,T}(t)}{|q_{\xi,T}(t)|},
& 0<t\le T,\quad q_{\xi,T}(t)\not=0,\\[2mm]
\displaystyle\lim_{s\uparrow t}
\dfrac{q_{\xi,T}(s)}{|q_{\xi,T}(s)|},
& 0<t\le T,\quad q_{\xi,T}(t)=0,\\[2mm]
0,
& t>T.
\end{cases}
\end{equation}
Then $u_{\xi,T}$ is piecewise continuous and left-continuous, with finitely
many discontinuities. Define
\begin{equation}\label{wang17}
x_{\xi,T}:=-\int_0^T e^{-As}Bu_{\xi,T}(s)\,ds.
\end{equation}

The following realization lemma will be used in the proof of
\textnormal{(i)}$\Rightarrow$\textnormal{(ii)} in
Theorem~\ref{wang5}.

\begin{lemma}[Realization of a terminal adjoint direction]\label{wang18}
The initial state $x_{\xi,T}$ defined by \eqref{wang17} satisfies
\begin{equation}\label{wang19}
T_0^*(x_{\xi,T})=T, \qquad u_{0,x_{\xi,T}}^*(t)=u_{\xi,T}(t), \qquad 0<t\le T,
\end{equation}
where $u_{\xi,T}$ is given by \eqref{wang16}.
\end{lemma}
\begin{proof}
We first prove the first equality in \eqref{wang19}.
By \eqref{wang16} and \eqref{wang17},
$u_{\xi,T}\in\mathcal U_{ad}$ and
\begin{equation}\label{wang20}
z(T;x_{\xi,T},u_{\xi,T})=0.
\end{equation}
Hence $u_{\xi,T}$ is admissible for $(P_{0,x_{\xi,T}})$.

For $t>0$, let
\begin{equation}\label{wang21}
\mathcal C(t):=
\left\{
-\int_0^t e^{-As}Bu(s)\,ds:
\|u\|_{L^\infty(0,t;\mathbb{R}^m)}\leq 1
\right\},
\end{equation}
which  is the set of initial states that can be steered to the
origin at time $t$. The support function  of $\mathcal C(t)$ is defined by
\begin{equation}\label{wang22}
h_t(p):=\sup_{y\in\mathcal C(t)}\langle p,y\rangle, \qquad p\in\mathbb{R}^n.
\end{equation}
By \eqref{wang21} and \eqref{wang22},
\begin{equation}\label{wang23}
h_t(p)=\int_0^t |B^\top e^{-A^\top s}p|\,ds, \qquad p\in\mathbb{R}^n,\;t>0.
\end{equation}

Put $p:=-e^{A^\top T}\xi$. Then
\begin{equation}\label{wang24}
-B^\top e^{-A^\top s}p=B^\top e^{A^\top(T-s)}\xi=q_{\xi,T}(s),
\qquad 0<s<T.
\end{equation}
By \eqref{wang16}, \eqref{wang17}, \eqref{wang23}, and \eqref{wang24}, we have
\begin{equation}\label{wang25}
\langle p,x_{\xi,T}\rangle=\int_0^T \langle q_{\xi,T}(s),u_{\xi,T}(s)\rangle\,ds
=\int_0^T |q_{\xi,T}(s)|\,ds=h_T(p).
\end{equation}
Since $q_{\xi,T}$ has only finitely many zeros in $[0,T]$,
\eqref{wang23} and
\eqref{wang24} give
\begin{equation}\label{wang26}
h_T(p)-h_t(p)=\int_t^T |q_{\xi,T}(s)|\,ds>0, \qquad 0<t<T.
\end{equation}
It follows from \eqref{wang25} and \eqref{wang26} that
$\langle p,x_{\xi,T}\rangle>h_t(p)$ for $0<t<T$. 
Hence, by \eqref{wang22},
\begin{equation}\label{wang27}
x_{\xi,T}\notin\mathcal C(t),\qquad 0<t<T.
\end{equation}
On the other hand, \eqref{wang20} and \eqref{wang21} yield
$x_{\xi,T}\in\mathcal C(T)$. Together with \eqref{wang27}, this gives
\begin{equation}\label{wang28}
T_0^*(x_{\xi,T})=T.
\end{equation}

By \eqref{wang20} and \eqref{wang28},
$u_{\xi,T}$ is an optimal control for $(P_{0,x_{\xi,T}})$. By
 Proposition~\ref{prop:basic-min-time}(i) and (iii),
the chosen optimal-control representative is unique, and hence
$u_{0,x_{\xi,T}}^*(t)=u_{\xi,T}(t)$ for every $t\in(0,T]$.
This, along with \eqref{wang28}, yields \eqref{wang19} and completes the proof.
\end{proof}

\section{Structural lemmas and proof of Theorem~\ref{wang5}}\mbox{}\par

\medskip

\noindent 

The proof of Theorem~\ref{wang5} is based on the following
lemmas. The first three lemmas prove that (i) implies
(ii). The fourth lemma proves that (ii)
implies both (i) and (iii). The last two
lemmas prove that (iii) implies (ii).

\begin{lemma}\label{wang29}
Assume that \eqref{wang4} holds. Set
$W:=\textnormal{Ran} B^\top\subset\mathbb{R}^m$, and let $P_W$ denote the orthogonal projection
from $\mathbb{R}^m$ onto $W$. Then there exists $\beta>0$ such that
\begin{equation}\label{wang30}
B^\top B=\beta P_W.
\end{equation}
\end{lemma}

\begin{proof}
By \eqref{wang3}, $B\not=0$ and hence $W\not=\{0\}$. 
Fix an arbitrary unit vector $v\in W$ and $T>0$, and choose $\xi\in\mathbb{R}^n$
such that $B^\top\xi=v$.

\smallskip
\noindent{\it Step 1. We prove that
\begin{equation}\label{wang31}
T_0^*(x_{\xi,T})=T, \qquad \lim_{t\uparrow T}u_{0,x_{\xi,T}}^*(t)=v.
\end{equation}}

It follows from Lemma~\ref{wang18} that
\begin{equation}\label{wang32}
T_0^*(x_{\xi,T})=T,
\qquad
u_{0,x_{\xi,T}}^*(t)=u_{\xi,T}(t),
\qquad 0<t\le T.
\end{equation}
Since $B^\top\xi=v\not=0$, we have
$B^\top e^{A^\top(T-t)}\xi\not=0$
for $t<T$ sufficiently close to $T$. Hence, by \eqref{wang15},
\eqref{wang16}, and \eqref{wang32},
\begin{equation}\label{wang33}
u_{0,x_{\xi,T}}^*(t)=\frac{B^\top e^{A^\top(T-t)}\xi}
{|B^\top e^{A^\top(T-t)}\xi|}
\end{equation}
for $t<T$ sufficiently close to $T$. Letting $t\uparrow T$ in
\eqref{wang33}, we obtain
\begin{equation*}
\lim_{t\uparrow T}u_{0,x_{\xi,T}}^*(t)
=\frac{B^\top\xi}{|B^\top\xi|}=v.
\end{equation*}
Together with the first equality in \eqref{wang32}, this proves
\eqref{wang31}.

\smallskip
\noindent{\it Step 2. We prove that, if $x_\varepsilon$ is the terminal state of
$(P_{\varepsilon,x_{\xi,T}})$, then
\begin{equation}\label{wang34}
\eta_\varepsilon:=\frac{x_\varepsilon}{\varepsilon}
\longrightarrow
\eta:=-\frac{Bv}{|Bv|}
\qquad\text{as }\varepsilon\downarrow 0.
\end{equation}}

Choose $\varepsilon_0>0$ sufficiently small such that
\eqref{wang4} holds for $x_{\xi,T}$ whenever
$0<\varepsilon<\varepsilon_0$, and $\varepsilon_0<|x_{\xi,T}|$. For $0<\varepsilon<\varepsilon_0$, set
$
T_\varepsilon:=T_\varepsilon^*(x_{\xi,T})
$ and 
$
\tau_\varepsilon:=T-T_\varepsilon.
$
By \eqref{wang31}, $z_{0,x_{\xi,T}}^*(T)=0$. Since
$\varepsilon<|x_{\xi,T}|$, the continuity of $z_{0,x_{\xi,T}}^*$ shows that it
enters $\overline{B}_\varepsilon(0)$ at some time strictly smaller than $T$. Hence
\begin{equation}\label{wang35}
T_\varepsilon<T,
\qquad
\tau_\varepsilon>0.
\end{equation}
Moreover, by Proposition~\ref{prop:basic-min-time}(iv) and
\eqref{wang31},
$T_\varepsilon^*(x_{\xi,T})\to T$ as $\varepsilon\downarrow0$. Hence
\begin{equation}\label{wang36}
\tau_\varepsilon\longrightarrow 0
\qquad\text{as }\varepsilon\downarrow 0.
\end{equation}

Next, it follows from \eqref{wang4} that
\begin{equation*}
u_{\varepsilon,x_{\xi,T}}^*(t)
=u_{0,x_{\xi,T}}^*(t),
\qquad 0<t\leq T_\varepsilon.
\end{equation*}
Let $x_\varepsilon$ be the terminal state of $(P_{\varepsilon,x_{\xi,T}})$. By the
above equality, the uniqueness of solutions of \eqref{wang1}, and 
Proposition~\ref{prop:basic-min-time}(i),
\begin{equation}\label{wang37}
x_\varepsilon
=z_{\varepsilon,x_{\xi,T}}^*(T_\varepsilon)
=z_{0,x_{\xi,T}}^*(T_\varepsilon),
\qquad
|x_\varepsilon|=\varepsilon.
\end{equation}
Since $z_{0,x_{\xi,T}}^*(T)=0$, by \eqref{wang37} and the
variation-of-constants formula,
\begin{equation}\label{wang38}
x_\varepsilon
=-\int_0^{\tau_\varepsilon}e^{-Ar}B
u_{0,x_{\xi,T}}^*(T_\varepsilon+r)\,dr.
\end{equation}

Moreover, by \eqref{wang36} and \eqref{wang31},
\begin{equation}\label{wang39}
\sup_{0\leq r\leq\tau_\varepsilon}
\left|e^{-Ar}B u_{0,x_{\xi,T}}^*(T_\varepsilon+r)-Bv\right|
\longrightarrow0
\qquad\text{as }\varepsilon\downarrow 0.
\end{equation}
It follows from \eqref{wang38},
\eqref{wang39}, and \eqref{wang35} that
\begin{equation}\label{wang40}
\frac{x_\varepsilon}{\tau_\varepsilon}
\longrightarrow -Bv
\qquad\text{as }\varepsilon\downarrow 0.
\end{equation}
By \eqref{wang37} and \eqref{wang40},
\begin{equation}\label{wang41}
\frac{\varepsilon}{\tau_\varepsilon}
=\frac{|x_\varepsilon|}{\tau_\varepsilon}
\longrightarrow |Bv|
\qquad\text{as }\varepsilon\downarrow0.
\end{equation}
Since $v\in W=\mbox{Ran} B^\top=(\mbox{Ker} B)^\perp$ and $v\not=0$, we have $Bv\not=0$. Hence, by
\eqref{wang40} and \eqref{wang41}, we obtain
\eqref{wang34}.

\smallskip
\noindent{\it Step 3. We prove that}
\begin{equation}\label{wang42}
v
=
\frac{B^\top Bv}{|B^\top Bv|}.
\end{equation}


By \eqref{wang34}, we have
\begin{equation}\label{wang43}
B^\top\eta
=-\frac{B^\top Bv}{|Bv|}.
\end{equation}
Since $Bv\not=0$ and $\langle B^\top Bv,v\rangle=|Bv|^2>0$,
we obtain $B^\top Bv\not=0$. Hence, by \eqref{wang43},
$B^\top\eta\not=0$. Shrinking $\varepsilon_0$ if necessary,
\eqref{wang34} gives
\begin{equation*}
B^\top\eta_\varepsilon\not=0,
\qquad 0<\varepsilon<\varepsilon_0.
\end{equation*}

For $0<\varepsilon<\varepsilon_0$, the normalized terminal state in
 Proposition~\ref{prop:basic-min-time}(ii) is precisely
$\eta_\varepsilon$. Since $B^\top\eta_\varepsilon\not=0$, the formula in 
Proposition~\ref{prop:basic-min-time}(ii) and the left
continuity specified in (iii) give
\begin{equation}\label{wang44}
u_{\varepsilon,x_{\xi,T}}^*(T_\varepsilon)
=-\frac{B^\top\eta_\varepsilon}{|B^\top\eta_\varepsilon|},
\qquad 0<\varepsilon<\varepsilon_0.
\end{equation}
On the other hand, by \eqref{wang4}, 
$u_{0,x_{\xi,T}}^*(T_\varepsilon)
=u_{\varepsilon,x_{\xi,T}}^*(T_\varepsilon)$ 
for $0<\varepsilon<\varepsilon_0$. 
Together with \eqref{wang44}, this yields
\begin{equation}\label{wang45}
u_{0,x_{\xi,T}}^*(T_\varepsilon)
=-\frac{B^\top\eta_\varepsilon}{|B^\top\eta_\varepsilon|},
\qquad 0<\varepsilon<\varepsilon_0.
\end{equation}

Letting $\varepsilon\downarrow0$ in \eqref{wang45} and using
\eqref{wang36}, \eqref{wang31}, and
\eqref{wang34}, we obtain
\begin{equation*}
v=-\frac{B^\top\eta}{|B^\top\eta|}
=\frac{B^\top Bv}{|B^\top Bv|}.
\end{equation*}
This proves \eqref{wang42}.

\smallskip
\noindent{\it Step 4. We prove that there exists $\beta>0$ such that}
\begin{equation}\label{wang46}
B^\top Bw=\beta w
\qquad\mbox{for every }w\in W.
\end{equation}

Since the unit vector $v\in W$ was arbitrary,
\eqref{wang42} holds for every unit vector in $W$. Hence,
for each $w\in W\setminus\{0\}$, there exists $\lambda_w>0$ such that
\begin{equation}\label{wang47}
B^\top Bw=\lambda_w w.
\end{equation}

We now claim that $\lambda_w$ is independent of $w$. Let
$w_1,w_2\in W\setminus\{0\}$. If $w_1$ and $w_2$ are linearly dependent,
then \eqref{wang47} and the linearity of $B^\top B$ give
$\lambda_{w_1}=\lambda_{w_2}$. If they are linearly independent, then
\begin{equation}\label{wang48}
B^\top B(w_1+w_2)
=\lambda_{w_1+w_2}(w_1+w_2).
\end{equation}
On the other hand, by \eqref{wang47}, 
$B^\top B(w_1+w_2)
=\lambda_{w_1}w_1+\lambda_{w_2}w_2$. 
Comparing this with \eqref{wang48} and using the linear
independence of $w_1$ and $w_2$, we obtain $\lambda_{w_1}=\lambda_{w_2}$.
Hence, $\lambda_w$ is independent of $w$.

The above claim, along with \eqref{wang47}, implies \eqref{wang46} for some
  $\beta>0$.

\smallskip
\noindent{\it Step 5. We prove \eqref{wang30}.}

We have $W^\perp=(\mbox{Ran} B^\top)^\perp=\mbox{Ker} B$. 
Let $x\in\mathbb{R}^m$ be arbitrary and write
$x=P_Wx+(I-P_W)x$.
Since $P_Wx\in W$, it follows from \eqref{wang46} that
$B^\top B(P_Wx)=\beta P_W x$. 
Since $(I-P_W)x\in W^\perp=\mbox{Ker} B$, we have $B^\top B((I-P_W)x)=0$.
Therefore,
\begin{equation*}
B^\top Bx=B^\top B(P_Wx)+B^\top B((I-P_W)x)
=\beta P_Wx.
\end{equation*}
Since $x\in\mathbb{R}^m$ was arbitrary, \eqref{wang30} follows.
This completes the proof of the lemma.
\end{proof}

\begin{lemma}\label{wang49}
Assume that \eqref{wang4} holds, and let $\beta>0$ be the
constant in \eqref{wang30}. Then
\begin{equation}\label{wang50}
BB^\top=\beta I_n.
\end{equation}
\end{lemma}
\begin{proof}
Suppose, by contradiction, that \eqref{wang50} does not hold.

\smallskip
\noindent{\it Step 1. We prove that}
\begin{equation}\label{wang51}
\mbox{Ker} B^\top\not=\{0\}.
\end{equation}

By contradiction, we suppose that \eqref{wang51} does not hold. Then  $B^\top$ is injective.
Since $B^\top y\in W$ for every
$y\in\mathbb{R}^n$, it follows from \eqref{wang30} that 
$B^\top BB^\top y=\beta B^\top y$. 
This, along with the injectivity of $B^\top$, implies
$BB^\top y=\beta y$ for every $y\in\mathbb{R}^n$,
which leads to a contradiction. Hence, \eqref{wang51} holds.

We now make some preparations for the next step.
Fix
\begin{equation}\label{wang52}
w\in\mbox{Ker} B^\top\setminus\{0\}.
\end{equation}
By \eqref{wang3}, $B\not=0$. Hence, after a normalization, we can choose
$\xi_0\in\mathbb{R}^n$ such that
\begin{equation}\label{wang53}
v:=B^\top\xi_0,
\qquad |v|=1.
\end{equation}
For $\alpha\in\mathbb{R}$, set
$\xi_\alpha:=\xi_0+\alpha w$.
Since $B^\top w=0$, by \eqref{wang53},
\begin{equation}\label{wang54}
B^\top\xi_\alpha=v,
\qquad \alpha\in\mathbb{R}.
\end{equation}
In particular,
\begin{equation}\label{wang55}
\xi_\alpha\ne0,
\qquad \alpha\in\mathbb{R}.
\end{equation}

Fix $T>0$ and set
\begin{equation*}
u_\alpha(t):=u_{0,x_{\xi_\alpha,T}}^*(t),
\qquad 0<t\le T,\quad \alpha\in\mathbb{R}.
\end{equation*}

\smallskip
\noindent{\it Step 2. We prove that}
\begin{equation}\label{wang56}
T_0^*(x_{\xi_\alpha,T})=T,
\qquad \alpha\in\mathbb{R},
\end{equation}
\noindent{\it and}
\begin{equation}\label{wang57}
\lim_{t\uparrow T}u_\alpha(t)=v,
\qquad \alpha\in\mathbb{R}.
\end{equation}



Indeed, by (\ref{wang54}) and the same argument used to prove (\ref{wang31}), with $\xi$ 
replaced by $\xi_\alpha$, we have (\ref{wang56}) and (\ref{wang57}).

\smallskip
\noindent{\it Step 3. Let $\eta_{\varepsilon,\alpha}$ denote the normalized
terminal state of $(P_{\varepsilon,x_{\xi_\alpha,T}})$, and set 
$T_{\varepsilon,\alpha}:=T_\varepsilon^*(x_{\xi_\alpha,T})$. 
We prove that}
\begin{equation}\label{wang59}
\eta_{\varepsilon,\alpha}
\longrightarrow
\eta:=-\frac{Bv}{|Bv|}
\qquad\text{as }\varepsilon\downarrow0,
\end{equation}
\noindent{\it and that}
\begin{equation}\label{wang60}
T_{\varepsilon,\alpha}\longrightarrow T
\qquad\text{as }\varepsilon\downarrow0.
\end{equation}




Indeed, by the same arguments used to prove \eqref{wang34} and \eqref{wang36},  with $\xi$ 
replaced by $\xi_\alpha$, we 
obtain (\ref{wang59}) and (\ref{wang60}).

\smallskip

\noindent{\it Step 4. With $\eta$ given by \eqref{wang59}, set 
$q(t):=B^\top e^{A^\top(T-t)}\eta,\;\;t\in\mathbb{R}$. 
We prove that, for every $t\in(0,T)$ with $q(t)\not=0$,
\begin{equation}\label{wang61}
u_\alpha(t)
=-\frac{q(t)}{|q(t)|},
\qquad \alpha\in\mathbb{R}.
\end{equation}}

By \eqref{wang3} and $\eta\ne0$, $q(\cdot)$ is not identically zero.
Since $q(\cdot)$ is real analytic, it has only finitely many zeros in $[0,T]$.
Fix an arbitrary $t\in(0,T)$ with $q(t)\not=0$ and an arbitrary
$\alpha\in\mathbb{R}$.
By these facts, together with \eqref{wang59} and \eqref{wang60},
we may choose $\varepsilon_0>0$ sufficiently small such that, for
$0<\varepsilon<\varepsilon_0$,
\begin{equation*}
t<T_{\varepsilon,\alpha},\;\;\;\;B^\top e^{A^\top(T_{\varepsilon,\alpha}-t)}
\eta_{\varepsilon,\alpha}\ne0.
\end{equation*}

By Proposition~\ref{prop:basic-min-time}(ii) and the
choice of $\varepsilon_0$, we have
\begin{equation}\label{wang62}
u_{\varepsilon,x_{\xi_\alpha,T}}^*(t)
=
-\frac{B^\top e^{A^\top(T_{\varepsilon,\alpha}-t)}
\eta_{\varepsilon,\alpha}}
{|B^\top e^{A^\top(T_{\varepsilon,\alpha}-t)}
\eta_{\varepsilon,\alpha}|}.
\end{equation}
On the other hand, shrinking $\varepsilon_0$ if necessary, \eqref{wang4} also holds
for $x_{\xi_\alpha,T}$ whenever $0<\varepsilon<\varepsilon_0$. Thus,  by \eqref{wang4},
\begin{equation}\label{wang63}
u_\alpha(t)
=u_{\varepsilon,x_{\xi_\alpha,T}}^*(t).
\end{equation}
Combining \eqref{wang62} and
\eqref{wang63}, we obtain 
\begin{equation*}
u_\alpha(t)
=-\frac{B^\top e^{A^\top(T_{\varepsilon,\alpha}-t)}
\eta_{\varepsilon,\alpha}}
{|B^\top e^{A^\top(T_{\varepsilon,\alpha}-t)}
\eta_{\varepsilon,\alpha}|}.
\end{equation*}
Letting $\varepsilon\downarrow0$ and using \eqref{wang59} and
\eqref{wang60}, we obtain 
$u_\alpha(t)=-q(t)/|q(t)|$. 
Since $t\in(0,T)$ with $q(t)\not=0$ and $\alpha\in\mathbb{R}$ were arbitrary,
this proves \eqref{wang61}.

\smallskip
\noindent{\it Step 5. Let $w$ be given by \eqref{wang52}. We prove that}
\begin{equation}\label{wang64}
B^\top e^{A^\top(T-t)}w=0,
\qquad 0<t<T,
\end{equation}
\noindent{\it and then derive a contradiction.}

Set
\begin{equation*}
a_0(s):=B^\top e^{A^\top(T-s)}\xi_0,
\qquad
b_0(s):=B^\top e^{A^\top(T-s)}w,
\qquad 0<s<T.
\end{equation*}
By \eqref{wang15}, \eqref{wang16}, and \eqref{wang19}, we have
\begin{equation}\label{2026-9-6-1}
u_\alpha(t)
=\frac{B^\top e^{A^\top(T-t)}\xi_\alpha}
{|B^\top e^{A^\top(T-t)}\xi_\alpha|},
\end{equation}
when $t\in(0,T)$ satisfies $B^\top e^{A^\top(T-t)}\xi_\alpha\not=0$.
Fix an arbitrary $t\in(0,T)$ with $q(t)\ne0$. Since
$\xi_\alpha=\xi_0+\alpha w,\;\alpha\in\mathbb{R}$, we have
\begin{equation*}
B^\top e^{A^\top(T-t)}\xi_\alpha
=a_0(t)+\alpha b_0(t),
\qquad \alpha\in\mathbb{R}.
\end{equation*}
Hence, by \eqref{2026-9-6-1} and \eqref{wang61},
\begin{equation}\label{wang65}
\frac{a_0(t)+\alpha b_0(t)}
{|a_0(t)+\alpha b_0(t)|}
=
-\frac{q(t)}{|q(t)|}
\end{equation}
for every $\alpha\in\mathbb{R}$ such that
$a_0(t)+\alpha b_0(t)\not=0$. 

Suppose that $b_0(t)\not=0$. Then
$a_0(t)+\alpha b_0(t)\ne0$ for all sufficiently large $|\alpha|$.
Letting $\alpha\to+\infty$ in
\eqref{wang65}, we obtain
$b_0(t)/|b_0(t)|=-q(t)/|q(t)|$. 
On the other hand, letting $\alpha\to-\infty$ in
\eqref{wang65}, we obtain
$-b_0(t)/|b_0(t)|=-q(t)/|q(t)|$, 
which is a contradiction.
Thus, we have
\begin{equation}\label{wang66}
b_0(t)=0
\qquad
\text{for every }t\in(0,T)\text{ such that }q(t)\not=0.
\end{equation}

Since $q(\cdot)$ has only finitely many zeros in $[0,T]$ and $b_0$ is
continuous, \eqref{wang66} implies
\begin{equation*}
b_0(t)=0,
\qquad 0<t<T.
\end{equation*}
By the definition of $b_0$, this is precisely
\eqref{wang64}. Thus \eqref{wang64} is proved. This 
contradicts the fact that the function $B^\top e^{A^\top(T-\cdot)}w$ has only finitely many 
zeros in $[0,T]$. Thus \eqref{wang50} holds, completing the proof.
\end{proof}

\begin{lemma}\label{wang69}
Assume that \eqref{wang4} holds. Then there exists $a\le0$
such that
\begin{equation}\label{wang70}
A+A^\top=2aI_n.
\end{equation}
\end{lemma}
\begin{proof}
By Lemma~\ref{wang49}, \eqref{wang50} holds. Hence
\begin{equation}\label{wang71}
|B^\top y|=\sqrt{\beta}\,|y|
\qquad\mbox{for every }y\in\mathbb{R}^n,
\end{equation}
which implies that $B^\top$ is injective.

Fix an arbitrary unit vector $p\in\mathbb{R}^n$ and $T>0$. Applying
Lemma~\ref{wang18} with $\xi=p$, we have $T_0^*(x_{p,T})=T$.

\smallskip
\noindent{\it Step 1. Set}
\begin{equation}\label{wang72}
\rho(\tau):=
\frac{e^{A^\top\tau}p}{|e^{A^\top\tau}p|},
\qquad 0\le\tau<T.
\end{equation}
\noindent{\it We prove that}
\begin{equation}\label{wang73}
\rho(\tau)=p+\tau P_pA^\top p+o(\tau)
\qquad\text{as }\tau\downarrow 0,
\end{equation}
\noindent{\it where}
\begin{equation}\label{wang74}
P_p:=I_n-pp^\top.
\end{equation}

Differentiating \eqref{wang72} at $\tau=0$ and using $|p|=1$ and
\eqref{wang74}, we obtain
\begin{equation}\label{wang76}
\rho(0)=p,\qquad\rho'(0)=
A^\top p-\langle A^\top p,p\rangle p
=P_pA^\top p.
\end{equation}
By \eqref{wang76} and the Taylor expansion of $\rho$ at
$\tau=0$, we obtain \eqref{wang73}.

\smallskip
\noindent{\it Step 2. Set}
\begin{equation}\label{wang77}
x(\tau):=z_{0,x_{p,T}}^*(T-\tau),
\qquad 0\le\tau<T.
\end{equation}
\noindent{\it We prove that, for all sufficiently small $\tau>0$, the
normalized direction}
\begin{equation}\label{wang78}
\eta(\tau):=\frac{x(\tau)}{|x(\tau)|}
\end{equation}
\noindent{\it is well defined and satisfies
\begin{equation}\label{wang79}
\eta(\tau)
=-p+\frac{\tau}{2}P_p(A-A^\top)p+o(\tau)
\qquad\text{as }\tau\downarrow 0.
\end{equation}}

According to the optimality of $T$, (\ref{wang78}) is well defined.
By \eqref{wang72}, \eqref{wang15}, \eqref{wang16},
\eqref{wang19}, and \eqref{wang71},
\begin{equation}\label{wang75}
u_{0,x_{p,T}}^*(T-\tau)
=\frac{B^\top\rho(\tau)}{\sqrt{\beta}},
\qquad
Bu_{0,x_{p,T}}^*(T-\tau)
=\sqrt{\beta}\,\rho(\tau),
\qquad 0\leq\tau<T.
\end{equation}
Since $T_0^*(x_{p,T})=T$, we have $x(0)=0$. By
\eqref{wang1} and the second equality in
\eqref{wang75},
\begin{equation}\label{wang80}
x'(\tau)
=-Ax(\tau)-\sqrt{\beta}\,\rho(\tau),
\qquad 0\le\tau<T.
\end{equation}
Since $\rho(\cdot)$ is smooth near $\tau=0$, it follows from
\eqref{wang80}, together with $x(0)=0$ and $\rho(0)=p$, that
$x'(0)=-\sqrt{\beta}\,p$.
Differentiating \eqref{wang80} at $\tau=0$ and using
\eqref{wang76}, we further obtain
\begin{equation*}
x''(0)
=-Ax'(0)-\sqrt{\beta}\,\rho'(0)=\sqrt{\beta}\bigl(Ap-P_pA^\top p\bigr).
\end{equation*}
Thus, by the Taylor expansion of $x$ at $\tau=0$,
\begin{equation}\label{wang81}
x(\tau)
=-\sqrt{\beta}\,\tau p
+\frac{\sqrt{\beta}\,\tau^2}{2}
\bigl(Ap-P_pA^\top p\bigr)
+o(\tau^2)
\qquad\text{as }\tau\downarrow 0.
\end{equation}
For $0<\tau<T$, set
\begin{equation}\label{wang82}
y(\tau):=\frac{x(\tau)}{\sqrt{\beta}\tau}.
\end{equation}
By \eqref{wang81} and \eqref{wang82},
\begin{equation}\label{wang83}
y(\tau)
=
-p+\frac{\tau}{2}\bigl(Ap-P_pA^\top p\bigr)+o(\tau)
\qquad\text{as }\tau\downarrow0.
\end{equation}

Set $d:=Ap-P_pA^\top p\in\mathbb{R}^n$. Then, by \eqref{wang83},
\begin{equation}\label{2026-9-6-2}
y(\tau)
=-p+\frac{\tau}{2}d+o(\tau)
\qquad\text{as }\tau\downarrow0.
\end{equation}
Since $|p|=1$,
\begin{equation*}
|y(\tau)|^2
=1-\tau\langle p,d\rangle+o(\tau)
\qquad\text{as }\tau\downarrow 0.
\end{equation*}
Hence, by the first-order Taylor expansions of $\sqrt{1+s}$ and
$(1+s)^{-1}$ at $s=0$,
\begin{equation}\label{wang84}
|y(\tau)|
=1-\frac{\tau}{2}\langle p,d\rangle+o(\tau),
\qquad
\frac{1}{|y(\tau)|}
=1+\frac{\tau}{2}\langle p,d\rangle+o(\tau)
\qquad\text{as }\tau\downarrow 0.
\end{equation}
By \eqref{wang82} and \eqref{wang78}, 
$\eta(\tau)=y(\tau)/|y(\tau)|$
for all sufficiently small $\tau>0$. Hence, by
\eqref{2026-9-6-2} and \eqref{wang84},
\begin{equation*}
\eta(\tau)
=-p+\frac{\tau}{2}
\bigl(d-\langle p,d\rangle p\bigr)+o(\tau)=-p+\frac{\tau}{2}P_pd+o(\tau),
\end{equation*}
where the last equality follows from \eqref{wang74}.
Substituting $d=Ap-P_pA^\top p$ gives
\begin{equation}\label{wang85}
\eta(\tau)
=-p+\frac{\tau}{2}
P_p\bigl(Ap-P_pA^\top p\bigr)+o(\tau)
\qquad\text{as }\tau\downarrow 0.
\end{equation}
Finally, by \eqref{wang74} and $|p|=1$, we have $P_p^2=P_p$.
Together with \eqref{wang85}, this proves
\eqref{wang79}.

\smallskip
\noindent{\it Step 3. For $\varepsilon>0$, set 
$T_\varepsilon:=T_\varepsilon^*(x_{p,T})$, 
$\tau_\varepsilon:=T-T_\varepsilon$.
 We prove that, for all sufficiently small $\varepsilon>0$,}
\begin{equation}\label{wang86}
\rho(\tau_\varepsilon)=-\eta(\tau_\varepsilon).
\end{equation}

By the same argument used to prove \eqref{wang35} and
\eqref{wang36}, with $\xi$ replaced by $p$, we have
\begin{equation}\label{wang87}
\tau_\varepsilon>0,
\qquad
\tau_\varepsilon\longrightarrow 0
\qquad\text{as }\varepsilon\downarrow 0.
\end{equation}
Hence, for all sufficiently small $\varepsilon>0$, \eqref{wang87} ensures
that $\eta(\tau_\varepsilon)$ is well defined. We now fix an arbitrary such
$\varepsilon>0$.

By \eqref{wang77}, \eqref{wang4}, and the uniqueness of solutions of
\eqref{wang1}, the terminal state of $(P_{\varepsilon,x_{p,T}})$ satisfies
\begin{equation*}
z_{\varepsilon,x_{p,T}}^*(T_\varepsilon)
=z_{0,x_{p,T}}^*(T_\varepsilon)
=x(\tau_\varepsilon).
\end{equation*}
By Proposition~\ref{prop:basic-min-time}(i),
$|x(\tau_\varepsilon)|=\varepsilon$. Hence, by \eqref{wang78},
$x(\tau_\varepsilon)=\varepsilon\eta(\tau_\varepsilon)$. Proposition~\ref{prop:basic-min-time}(ii), together with \eqref{wang71} and
$|\eta(\tau_\varepsilon)|=1$, then gives
\begin{equation}\label{wang88}
u_{\varepsilon,x_{p,T}}^*(T_\varepsilon)
=
-\frac{B^\top\eta(\tau_\varepsilon)}{\sqrt{\beta}}.
\end{equation}
On the other hand, by $T_\varepsilon=T-\tau_\varepsilon$ and
\eqref{wang75},
\begin{equation}\label{wang89}
u_{0,x_{p,T}}^*(T_\varepsilon)
=
\frac{B^\top\rho(\tau_\varepsilon)}{\sqrt{\beta}}.
\end{equation}
By \eqref{wang4}, the left-hand sides of
\eqref{wang88} and \eqref{wang89} are equal. Hence
$B^\top\rho(\tau_\varepsilon)=-B^\top\eta(\tau_\varepsilon)$. 
Since $B^\top$ is injective by \eqref{wang71}, we obtain
\eqref{wang86}. Since $\varepsilon>0$ was arbitrary, \eqref{wang86} holds
for all sufficiently small $\varepsilon>0$.

\smallskip
\noindent{\it Step 4. We prove that}
\begin{equation}\label{wang90}
P_p(A+A^\top)p=0.
\end{equation}

By \eqref{wang73} and \eqref{wang79},
\begin{equation*}
\rho(\tau)+\eta(\tau)
=
\frac{\tau}{2}P_p(A+A^\top)p+o(\tau)
\qquad\text{as }\tau\downarrow 0.
\end{equation*}
By \eqref{wang87}, $\tau_\varepsilon>0$ and
$\tau_\varepsilon\rightarrow 0$ as $\varepsilon\downarrow 0$. Hence, evaluating the above
expansion at $\tau=\tau_\varepsilon$ and using \eqref{wang86}, we obtain
$0=\tau_\varepsilon P_p(A+A^\top)p/2+o(\tau_\varepsilon)$.
Dividing by $\tau_\varepsilon$ and letting $\varepsilon\downarrow 0$, we obtain
\eqref{wang90}.

\smallskip
\noindent{\it Step 5. We prove \eqref{wang70}.}

Since $p\in\mathbb{R}^n$ was an arbitrary unit vector, by
\eqref{wang74} and \eqref{wang90},
$(A+A^\top)p$ is parallel to $p$ for every unit vector $p\in\mathbb{R}^n$.
By the same linear-algebra argument used to prove
\eqref{wang46}, there exists $a\in\mathbb{R}$ such that
$A+A^\top=2aI_n$. 
Writing $A=aI_n+S$, $S^\top=-S$, 
every eigenvalue of $A$ has real part $a$. By \eqref{wang2}, $a\leq 0$.
Therefore \eqref{wang70} holds, completing the proof.
\end{proof}

\begin{lemma}\label{wang91}
Assume \textnormal{(ii)}. Then \textnormal{(iii)} and
\textnormal{(i)} hold. In \textnormal{(iii)}, one may take $\Psi=\Phi$,
where
\begin{equation}\label{wang92}
\Phi(r):=\int_0^r\frac{d\xi}{\sqrt{\beta}-a\xi},
\qquad r\ge0,
\end{equation}
and $a\leq 0$ and $\beta>0$ are given by \eqref{wang6}.
Moreover, for every $x\ne0$ and every $0<\varepsilon<|x|$,
\begin{equation}\label{wang93}
u_{\varepsilon,x}^*(t)=u_{0,x}^*(t),
\qquad 0<t\le T_\varepsilon^*(x).
\end{equation}
\end{lemma}

\begin{proof}
By \eqref{wang6}, we may write
\begin{equation}\label{wang94}
A=aI_n+S,
\qquad
S^\top=-S,\qquad a\leq 0.
\end{equation}

Let $x\not=0$ and $0\leq\varepsilon<|x|$ be arbitrary.

\smallskip
\noindent{\it Step 1. We prove that}
\begin{equation}\label{wang95}
T_\varepsilon^*(x)\geq\Phi(|x|)-\Phi(\varepsilon).
\end{equation}

Let $u$ be admissible for $(P_{\varepsilon,x})$ and set
\begin{equation}\label{wang96}
\tau_\varepsilon(u):=\min\{t>0:z(t;x,u)\in\overline{B}_\varepsilon(0)\}.
\end{equation}
Set
\begin{equation}\label{wang97}
r(t;u):=|z(t;x,u)|,\qquad t\geq 0.
\end{equation}
Then $r(\cdot;u)$ is absolutely continuous and, by
\eqref{wang96},
\begin{equation}\label{wang98}
r(t;u)>\varepsilon,\qquad 0\le t<\tau_\varepsilon(u);
\qquad
r(\tau_\varepsilon(u);u)=\varepsilon.
\end{equation}
Hence we can set
\begin{equation}\label{wang99}
\theta(t;u):=\frac{z(t;x,u)}{|z(t;x,u)|},
\qquad 0\leq t<\tau_\varepsilon(u).
\end{equation}
By \eqref{wang6}, \eqref{wang94},
\eqref{wang97}, and \eqref{wang99}, for a.e.
$t\in(0,\tau_\varepsilon(u))$,
\begin{equation*}
\frac{d}{dt}r(t;u)
=
ar(t;u)+\langle B^\top\theta(t;u),u(t)\rangle
\geq ar(t;u)-\sqrt{\beta}.
\end{equation*}
This implies that if we let $\Phi$ be given by \eqref{wang92}, then
\begin{equation*}
\frac{d}{dt}\Phi(r(t;u))\geq -1
\qquad\mbox{for a.e. }t\in(0,\tau_\varepsilon(u)).
\end{equation*}
Integrating it over $(0,\tau_\varepsilon(u))$ and using \eqref{wang97} and
\eqref{wang98}, we obtain 
$\tau_\varepsilon(u)\geq\Phi(|x|)-\Phi(\varepsilon)$. 
Applying this to $u=u_{\varepsilon,x}^*$ gives
\eqref{wang95}.

\smallskip
\noindent{\it Step 2. We  prove
\begin{equation}\label{wang100}
T_\varepsilon^*(x)=\Phi(|x|)-\Phi(\varepsilon).
\end{equation}}

Set
\begin{equation*}
\theta_0:=\frac{x}{|x|},
\qquad
\theta(t):=e^{St}\theta_0,\qquad t\geq 0,
\end{equation*}
and let $r_*$ solve
\begin{equation}\label{wang101}
r_*'(t)=ar_*(t)-\sqrt{\beta},
\qquad t\geq 0,
\qquad
r_*(0)=|x|.
\end{equation}
Since $a\leq 0$,  we have that $r_*'(t)\le-\sqrt{\beta}$ whenever $r_*(t)>0$.
Hence $r_*(\cdot)$ reaches zero in finite time. Set
\begin{equation}\label{wang102}
\tau_*:=\min\{t>0:r_*(t)=0\}.
\end{equation}
Define
\begin{equation}\label{wang103}
u_*(t):=-\frac{B^\top\theta(t)}{\sqrt{\beta}},
\quad 0\leq t\leq\tau_*;
\qquad
u_*(t)=0,\quad t>\tau_*,
\end{equation}
and
\begin{equation}\label{wang104}
z_*(t):=r_*(t)\theta(t),
\qquad 0\leq t\leq\tau_*.
\end{equation}
By \eqref{wang94}, $|\theta(t)|=1$. Hence, by
\eqref{wang6} and \eqref{wang103},
$u_*\in\mathcal U_{ad}$. Moreover, by \eqref{wang6},
\eqref{wang94}, \eqref{wang101}, and
\eqref{wang104},
\begin{equation*}
z_*'(t)=Az_*(t)+Bu_*(t),
\qquad 0\le t<\tau_*,\qquad z_*(0)=x.
\end{equation*}
Thus, by the uniqueness of solutions of
\eqref{wang1},
\begin{equation}\label{wang105}
z_*(t)=z(t;x,u_*),
\qquad 0\leq t\leq\tau_*.
\end{equation}

Meanwhile, it follows from \eqref{wang92} and \eqref{wang101} that
\begin{equation}\label{wang106}
\Phi(r_*(t))=\Phi(|x|)-t,
\qquad 0\leq t\leq\tau_*.
\end{equation}
By \eqref{wang102} and \eqref{wang106}, 
$\tau_*=\Phi(|x|)$.
Since $\Phi(\cdot)$ is strictly increasing, by \eqref{wang106}, the first
time at which $r_*(\cdot)$ reaches $\varepsilon$ is
\begin{equation}\label{wang107}
t_\varepsilon:=\Phi(|x|)-\Phi(\varepsilon).
\end{equation}
In particular, $0\leq t_\varepsilon\leq\tau_*$. Hence, by
\eqref{wang105},
\begin{equation}\label{wang108}
|z(t_\varepsilon;x,u_*)|=r_*(t_\varepsilon)=\varepsilon.
\end{equation}
Thus $u_*$ is admissible for $(P_{\varepsilon,x})$, and 
$T_\varepsilon^*(x)\leq t_\varepsilon=\Phi(|x|)-\Phi(\varepsilon)$. 
Together with \eqref{wang95}, this proves
\eqref{wang100}.

By \eqref{wang100},
\eqref{wang107}, and $\tau_*=\Phi(|x|)$,
\begin{equation}\label{wang109}
T_\varepsilon^*(x)=t_\varepsilon\leq T_0^*(x)=\tau_*.
\end{equation}

\smallskip
\noindent{\it Step 3. We prove \eqref{wang93}.}

Define
\begin{equation*}
\widetilde u_\varepsilon(t):=
\begin{cases}
u_*(t),&0<t\leq T_\varepsilon^*(x),\\
0,&t>T_\varepsilon^*(x).
\end{cases}
\end{equation*}
Since $u_*\in\mathcal U_{ad}$, we have $\widetilde u_\varepsilon\in\mathcal U_{ad}$.
Moreover, by \eqref{wang108},
\eqref{wang109}, and the uniqueness of solutions of
\eqref{wang1}, $|z(T_\varepsilon^*(x);x,\widetilde u_\varepsilon)|=\varepsilon$. 
Thus $\widetilde u_\varepsilon$ is optimal for $(P_{\varepsilon,x})$. Hence, by
the uniqueness of optimal control, \eqref{wang109}, and
\eqref{wang103},
\begin{equation}\label{wang110}
u_{\varepsilon,x}^*(t)
=-\frac{B^\top\theta(t)}{\sqrt{\beta}},
\qquad
0<t\leq T_\varepsilon^*(x).
\end{equation}

On the other hand, by \eqref{wang102},
\eqref{wang104}, \eqref{wang105}, and
\eqref{wang109}, $z(T_0^*(x);x,u_*)=z_*(\tau_*)=0$.
Thus $u_*$ is optimal for $(P_{0,x})$. Hence, by the uniqueness of optimal control and
\eqref{wang103},
\begin{equation*}
u_{0,x}^*(t)
=-\frac{B^\top\theta(t)}{\sqrt{\beta}},
\qquad
0<t\leq T_0^*(x).
\end{equation*}
Together with \eqref{wang110}, this proves
\eqref{wang93}.

\smallskip
\noindent{\it Step 4. We prove that \eqref{wang7} holds with
$\Psi=\Phi$.}

By \eqref{wang100} with $\varepsilon=0$,
\begin{equation*}
V_0(x)=T_0^*(x)=\Phi(|x|),
\qquad x\not=0.
\end{equation*}
Since $V_0(0)=\Phi(0)=0$, \eqref{wang7} holds with
$\Psi=\Phi$. This completes the proof.
\end{proof}

\begin{lemma}\label{wang111}
Assume \textnormal{(iii)}. Then, for every $T>0$, there exists $R(T)\geq 0$
such that
\begin{equation}\label{wang112}
\mathcal C(T)=\overline{B}_{R(T)}(0).
\end{equation}
\end{lemma}
\begin{proof}
Recall the controllable set $\mathcal C(T)$ defined in
\eqref{wang21}. If $x\in\mathcal C(T)$, then
$V_0(x)\leq T$. Conversely, if $V_0(x)\leq T$, extending an optimal
control by zero after $T_0^*(x)$ shows that $x\in\mathcal C(T)$.
Therefore
\begin{equation}\label{wang113}
\mathcal C(T)=\{x\in\mathbb{R}^n: V_0(x)\leq T\},
\qquad T>0.
\end{equation}
By \eqref{wang7} and \eqref{wang113}, we have
\begin{equation}\label{2026-9-6-4}
\mathcal C(T)=\{x\in\mathbb{R}^n: \Psi(|x|)\leq T\},
\qquad T>0.
\end{equation}
Since $\mathcal C(T)$ is compact, we can let
$R(T):=\max\{|x|:x\in\mathcal C(T)\}$. 
This, along with (\ref{2026-9-6-4}) and the convexity of $\mathcal C(T)$, implies
$\overline{B}_{R(T)}(0)\subset\mathcal C(T)$, 
while the reverse inclusion follows from the definition of $R(T)$.
Therefore \eqref{wang112} holds.
\end{proof}

\begin{lemma}\label{wang114}
Assume that \eqref{wang112} holds for every $T>0$. Then
\eqref{wang6} holds.
\end{lemma}
\begin{proof}
We first prove \eqref{wang50}.
Let $h_T(\cdot)$ be the support function defined in
\eqref{wang22}. By \eqref{wang112},
\begin{equation*}
h_T(p)=R(T)|p|,
\qquad p\in\mathbb{R}^n,\quad T>0.
\end{equation*}
Together with \eqref{wang23}, this gives
\begin{equation}\label{wang115}
\int_0^T|B^\top e^{-A^\top s}p|\,ds
=R(T)|p|,
\qquad p\in\mathbb{R}^n,\quad T>0.
\end{equation}
Hence, for any two unit vectors $p,q\in\mathbb{R}^n$,
\begin{equation*}
\int_0^T|B^\top e^{-A^\top s}p|\,ds
=\int_0^T|B^\top e^{-A^\top s}q|\,ds,
\qquad T>0.
\end{equation*}
Since the integrands are continuous, differentiating this equality with
respect to $T$ yields
\begin{equation}\label{wang116}
|B^\top e^{-A^\top T}p|
=|B^\top e^{-A^\top T}q|,
\qquad T>0.
\end{equation}
Thus, by \eqref{wang116}, for each $T>0$ the quadratic form
$p\longmapsto
p^\top e^{-AT}BB^\top e^{-A^\top T}p,\;p\in\mathbb{R}^n$ 
is constant on the unit sphere. To justify the resulting matrix identity,
set $E_T:=e^{-AT}BB^\top e^{-A^\top T}$ and let $\gamma(T)$ be the common
value of $p^\top E_T p$ for $|p|=1$. By homogeneity,
$p^\top E_T p=\gamma(T)|p|^2$ for every $p\in\mathbb{R}^n$. Since $E_T$ is real
symmetric, the polarization identity gives
$p^\top E_T q=\gamma(T)\langle p,q\rangle$ for all $p,q\in\mathbb{R}^n$.
Therefore
\begin{equation}\label{wang117}
e^{-AT}BB^\top e^{-A^\top T}
=\gamma(T)I_n,
\qquad T>0.
\end{equation}
Fix a unit vector $p\in\mathbb{R}^n$. It is clear that
\begin{equation*}
\gamma(T)=|B^\top e^{-A^\top T}p|^2
\longrightarrow |B^\top p|^2
\qquad\text{as }T\downarrow 0.
\end{equation*}
Hence, letting $T\downarrow 0$ in \eqref{wang117}, we obtain
$BB^\top=\beta I_n, \beta:=|B^\top p|^2\geq 0$. 
By \eqref{wang3}, $B\not=0$, and hence $\beta>0$. Therefore
\eqref{wang50} holds.

\smallskip

We next prove \eqref{wang70}. Taking traces in \eqref{wang117}, we obtain
\begin{equation*}
\gamma(T)=\frac{1}{n}
\operatorname{tr}\bigl(e^{-AT}BB^\top e^{-A^\top T}\bigr),
\qquad T>0.
\end{equation*}
The right-hand side defines a smooth extension of $\gamma$ to $T\ge0$,
and with this extension \eqref{wang117} holds also at $T=0$.
Differentiating \eqref{wang117} at $T=0$, we obtain
$-ABB^\top-BB^\top A^\top=\gamma'(0)I_n$.
Using \eqref{wang50}, this gives
$A+A^\top=-\gamma'(0)I_n/\beta$. 
Thus, setting
$a:=-\gamma'(0)/2\beta$, 
we obtain $A+A^\top=2aI_n$. 
Taking traces yields $a=\operatorname{tr}A/n$. 
Writing $A=aI_n+S$ with $S^\top=-S$, every eigenvalue of $A$ has real
part $a$. Hence, by \eqref{wang2}, $a\le0$. Therefore \eqref{wang70} holds. Together with
\eqref{wang50}, this yields \eqref{wang6}.
\end{proof}

\vskip 5pt
Finally, we prove Theorem~\ref{wang5}.
\begin{proof}[Proof of Theorem~\ref{wang5}]
Assume (i). Then \eqref{wang4} holds, and
Lemmas~\ref{wang29}--\ref{wang69} yield
\eqref{wang50} and \eqref{wang70}. Hence
\textnormal{(ii)} holds.

Assume (ii). By Lemma~\ref{wang91},
(i) and (iii) hold, and the ``Moreover'' assertion
of the theorem follows as well.

Finally, assume (iii). Lemma~\ref{wang111}
yields \eqref{wang112}, and Lemma~\ref{wang114} then yields
\eqref{wang6}. Hence (ii) holds. This completes the proof of Theorem~\ref{wang5}.
\end{proof}

\section{Proofs of the corollaries}

\begin{proof}[Proof of Corollary~\ref{wang8}]
If one of (i)--(iii) in
Theorem~\ref{wang5} holds, then \eqref{wang6} holds.
Taking traces in the two identities in \eqref{wang6}, we obtain
$\beta=\operatorname{tr}(BB^\top)/n, 
a=\operatorname{tr}A/n$, 
which is \eqref{wang10}.  By \eqref{wang100} and \eqref{wang92},
\begin{equation*}
T_\varepsilon^*(x)
=\int_\varepsilon^{|x|}
\frac{d\xi}{\sqrt{\beta}-a\xi},
\qquad
0\leq\varepsilon<|x|,
\end{equation*}
which is \eqref{wang9}. This completes the proof.
\end{proof}

\begin{proof}[Proof of Corollary~\ref{wang11}]
We divide the proof into three steps.

\medskip
\noindent{\it Step 1. We prove that $V_0\in C^{0,1}(\mathbb{R}^n)\cap C^\infty(\mathbb{R}^n\setminus\{0\})$
and that $V_0$ satisfies \eqref{wang12}.}

By \eqref{wang9} with $\varepsilon=0$ and \eqref{wang92}, we have
$V_0(x)=\Phi(|x|)$ for $x\not=0$.
By the definition of $V_0$ and \eqref{wang92}, we also have
$V_0(0)=\Phi(0)=0$.
Hence
\begin{equation*}
V_0(x)=\Phi(|x|),
\qquad x\in\mathbb{R}^n.
\end{equation*}
Since $a\leq 0$ by \eqref{wang10}, the above equality implies
that $V_0\in C^{0,1}(\mathbb{R}^n)\cap C^\infty(\mathbb{R}^n\setminus\{0\})$, and that
\begin{equation}\label{wang118}
\nabla V_0(x)
=\frac{x}{|x|(\sqrt{\beta}-a|x|)},\;\;x\not=0.
\end{equation}
Meanwhile, it follows from  \eqref{wang6} that
\begin{equation}\label{wang119}
\langle x,Ax\rangle=a|x|^2,
\qquad
|B^\top x|=\sqrt{\beta}\,|x|.
\end{equation}
Thus, by \eqref{wang118} and \eqref{wang119},
\begin{equation*}
1+\min_{|u|\leq 1}
\langle\nabla V_0(x),Ax+Bu\rangle=1+\frac{a|x|-\sqrt{\beta}}
{\sqrt{\beta}-a|x|}=0.
\end{equation*}
Therefore \eqref{wang12} holds.

\medskip
\noindent{\it Step 2. We prove that}
\begin{equation*}
u_{0,x}^*(t)=K(z_{0,x}^*(t)),
\qquad 0<t<T_0^*(x),
\end{equation*}
\noindent{\it and hence that $K$ is an optimal state feedback for
$(P_{0,x})$.}


Indeed, the assertions in Step 2 follow from the arguments in Step 2 and Step 3 of Lemma~\ref{wang91} directly.

\medskip
\noindent{\it Step 3. We prove that, for every $x\not=0$, $K(x)$ is the
unique minimizer in \eqref{wang14}.}

Fix $x\not=0$. By \eqref{wang118} and
\eqref{wang119}, $B^\top\nabla V_0(x)\not=0$.
Since $\langle\nabla V_0(x),Ax\rangle$ is independent of $v$, the unique
minimizer in \eqref{wang14} is
$-B^\top\nabla V_0(x)/|B^\top\nabla V_0(x)|$.
By \eqref{wang118} and \eqref{wang119}, this is
$-B^\top x/\sqrt{\beta}\,|x|=K(x)$. 
Thus (iii) of
Corollary~\ref{wang11} follows. This completes the proof.
\end{proof}

\begin{remark}[The nonlinear feedback]
{\it
Although the controlled dynamics is linear, the optimal feedback
$K(x)=-B^\top x/\sqrt{\beta}\,|x|$
is nonlinear and $0$-homogeneous: $K(\lambda x)=K(x)$ for every
$\lambda>0$. It is bounded on $\mathbb{R}^n\setminus\{0\}$, with $|K(x)|=1$,
but it does not admit a continuous extension to the origin.
}
\end{remark}

\section{Numerical illustrations}\label{sec:numerical}
\mbox{}\par

\medskip

\noindent The purpose of this section is
  to visualize the geometry behind exact truncation, to
separate the roles of the two rigidity identities in \eqref{wang6}, and to
illustrate a possible perturbative stability phenomenon near the rigid
class. All examples below satisfy (H1)--(H2).

In the numerical illustrations below, we use the truncation of optimal
trajectories to represent the exact truncation property, which was defined
in terms of optimal controls. We therefore first explain why these two
notions are equivalent in the present setting. We use throughout the
unique left-continuous optimal-control representatives fixed in
Proposition~\ref{prop:basic-min-time}.

\begin{definition}[Trajectory exact truncation]\label{wang120}
Let $x\not=0$ and $0<\varepsilon<|x|$. We say that {\it{trajectory exact
truncation}} holds for $(x,\varepsilon)$ if
\begin{equation}\label{wang121}
 z_{\varepsilon,x}^*(t)=z_{0,x}^*(t),
 \qquad 0\leq t\leq T_\varepsilon^*(x).
\end{equation}
\end{definition}

\begin{theorem}[Equivalence of control and trajectory exact truncation]
\label{wang122}
Let $x\not=0$ and $0<\varepsilon<|x|$. Then the following statements are equivalent:
\begin{enumerate}
\item[\textnormal{(i)}]
\begin{equation}\label{wang123}
 u_{\varepsilon,x}^*(t)=u_{0,x}^*(t),
 \qquad 0<t\leq T_\varepsilon^*(x);
\end{equation}
\item[\textnormal{(ii)}] Trajectory exact truncation holds for $(x,\varepsilon)$, namely,
\eqref{wang121} holds.
\end{enumerate}
\end{theorem}

\begin{proof}
Assume (i). Since the two state equations have the same initial
state and the same control on $(0,T_\varepsilon^*(x)]$, uniqueness of solutions of
\eqref{wang1} gives \eqref{wang121}.

Conversely, assume \eqref{wang121}. Then
$z_{0,x}^*(T_\varepsilon^*(x))=z_{\varepsilon,x}^*(T_\varepsilon^*(x))\in\overline{B}_\varepsilon(0)$.
Hence the restriction of $u_{0,x}^*$ to $(0,T_\varepsilon^*(x)]$, extended by zero
after $T_\varepsilon^*(x)$, reaches the ball at the optimal time $T_\varepsilon^*(x)$ and
is therefore an optimal control for $(P_{\varepsilon,x})$. By 
Proposition~\ref{prop:basic-min-time}(i) and (iii),
we have \eqref{wang123}. This completes the proof.
\end{proof}

For a fixed $x\not=0$ and $0<\varepsilon<|x|$, let
\begin{equation}\label{wang124}
\widehat T_\varepsilon(x)
:=
\min\{t>0:|z_{0,x}^*(t)|\leq\varepsilon\}
\end{equation}
be the first time at which $z_{0,x}^*(\cdot)$ enters $\overline{B}_\varepsilon(0)$. Define the {\it{truncation gap}}
\begin{equation}\label{wang125}
 G_\varepsilon(x):=\widehat T_\varepsilon(x)-T_\varepsilon^*(x)\geq 0.
\end{equation}
Indeed, the restriction of $u_{0,x}^*$ to $(0,\widehat T_\varepsilon(x)]$ is
admissible for $(P_{\varepsilon,x})$, so $T_\varepsilon^*(x)\leq\widehat T_\varepsilon(x)$.
Moreover, Theorem~\ref{wang122} gives
\begin{equation*}
 G_\varepsilon(x)=0
 \quad\Longleftrightarrow\quad
 \text{exact truncation holds for the pair $(x,\varepsilon)$}.
\end{equation*}
Thus the trajectory plots below give an equivalent geometric
representation of the exact truncation property studied in the paper.

\medskip
\noindent{\bf 1. The rigid case.}
Take
\begin{equation*}
 A=-\frac12I_2+\begin{pmatrix}0&-2\\2&0\end{pmatrix},
 \qquad B=I_2,
 \qquad x=(1,1)^\top.
\end{equation*}
Then $BB^\top=I_2$ and $A+A^\top=-I_2$. Figure~\ref{fig:num-rigid}
shows the optimal trajectory of $(P_{0,x})$ together with three tolerance
circles. The trajectory rotates while approaching the origin, but, for each
of the three values of $\varepsilon$, the optimal trajectory of $(P_{\varepsilon,x})$
is exactly the corresponding initial segment of the same curve. Thus the
skew-symmetric part of $A$ may rotate the optimal trajectory without
destroying exact truncation.

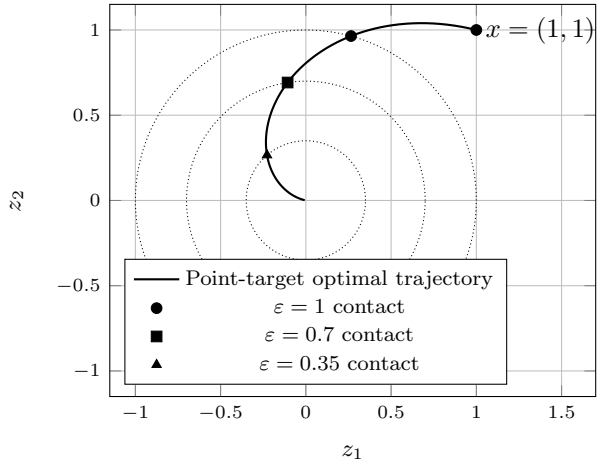
\begin{figure}[htbp]
\centering
\begin{tikzpicture}
\begin{axis}[
    width=.67\textwidth,
    height=.52\textwidth,
    axis equal image,
    tick label style={font=\scriptsize},
    label style={font=\small},
    xlabel={$z_1$},
    ylabel={$z_2$},
    xmin=-1.15,xmax=1.7,
    ymin=-1.15,ymax=1.15,
    grid=major,
    title={Rigid case: rotation is compatible with exact truncation},
     title style={font=\small},
    legend style={font=\footnotesize,at={(0.03,0.03)},anchor=south west}
]
\addplot[densely dotted,domain=0:360,samples=120,no marks,forget plot] ({cos(x)},{sin(x)});
\addplot[densely dotted,domain=0:360,samples=120,no marks,forget plot] ({0.7*cos(x)},{0.7*sin(x)});
\addplot[densely dotted,domain=0:360,samples=120,no marks,forget plot] ({0.35*cos(x)},{0.35*sin(x)});
\addplot[thick,no marks] coordinates {
(1.000000,1.000000) (0.956988,1.009927) (0.914281,1.018390) (0.871920,1.025415) (0.829949,1.031028) (0.788408,1.035257) (0.747338,1.038131) (0.706777,1.039681) (0.666762,1.039937) (0.627329,1.038931) (0.588512,1.036696) (0.550346,1.033265) (0.512862,1.028674) (0.476091,1.022957) (0.440061,1.016149) (0.404800,1.008288) (0.370336,0.999410) (0.336692,0.989553) (0.303892,0.978755) (0.271958,0.967054) (0.240911,0.954489) (0.210771,0.941100) (0.181554,0.926926) (0.153278,0.912007) (0.125956,0.896382) (0.099603,0.880092) (0.074230,0.863177) (0.049849,0.845678) (0.026468,0.827634) (0.004095,0.809086) (-0.017264,0.790074) (-0.037603,0.770639) (-0.056918,0.750820) (-0.075209,0.730656) (-0.092472,0.710188) (-0.108710,0.689455) (-0.123924,0.668495) (-0.138115,0.647347) (-0.151290,0.626049) (-0.163453,0.604639) (-0.174611,0.583155) (-0.184771,0.561631) (-0.193944,0.540106) (-0.202138,0.518614) (-0.209366,0.497191) (-0.215640,0.475870) (-0.220973,0.454685) (-0.225380,0.433670) (-0.228876,0.412856) (-0.231478,0.392275) (-0.233204,0.371957) (-0.234071,0.351932) (-0.234099,0.332229) (-0.233308,0.312877) (-0.231720,0.293902) (-0.229355,0.275330) (-0.226236,0.257188) (-0.222388,0.239499) (-0.217832,0.222287) (-0.212594,0.205574) (-0.206700,0.189382) (-0.200174,0.173731) (-0.193043,0.158641) (-0.185334,0.144130) (-0.177074,0.130215) (-0.168291,0.116913) (-0.159013,0.104238) (-0.149267,0.092205) (-0.139083,0.080827) (-0.128491,0.070116) (-0.117518,0.060083) (-0.106194,0.050738) (-0.094550,0.042088) (-0.082615,0.034143) (-0.070418,0.026907) (-0.057989,0.020388) (-0.045360,0.014589) (-0.032558,0.009513) (-0.019615,0.005163) (-0.006560,0.001539)
};
\addlegendentry{Point-target optimal trajectory}
\addplot[only marks,mark=*,forget plot] coordinates {(1,1)};
\addplot[only marks,mark=*] coordinates {(0.264860,0.964287)};
\addlegendentry{$\varepsilon=1$ contact}
\addplot[only marks,mark=square*] coordinates {(-0.106948,0.691782)};
\addlegendentry{$\varepsilon=0.7$ contact}
\addplot[only marks,mark=triangle*] coordinates {(-0.227804,0.265716)};
\addlegendentry{$\varepsilon=0.35$ contact}
\node[anchor=west] at (axis cs:1,1) {$x=(1,1)$};
\end{axis}
\end{tikzpicture}

\caption{A rigid system. The marked points are the optimal contacts with
$\partial B_\varepsilon(0)$ for $\varepsilon=1,0.7,0.35$; the corresponding ball-target
optimal trajectories are the initial segments ending at these points.}
\label{fig:num-rigid}
\end{figure}

\medskip
\noindent{\bf 2. Failure of the two isotropy conditions.}
First take
\begin{equation*}
 A=0,
 \qquad B=\operatorname{diag}(1,2),
 \qquad x=(1,1)^\top,
 \qquad \varepsilon=0.5.
\end{equation*}
Here $A+A^\top=0$, whereas $BB^\top=\operatorname{diag}(1,4)$ is not a scalar multiple of $I_2$.
 For $A=0$ and invertible $B$, the ball-target optimal time
is obtained from
\begin{equation*}
 T_\varepsilon^*(x)=\min_{|y|\leq\varepsilon}|B^{-1}(x-y)|.
\end{equation*}
Figure~\ref{fig:num-B} gives
\begin{equation*}
 T_\varepsilon^*(x)=0.672865,
 \qquad
 \widehat T_\varepsilon(x)=0.722749,
 \qquad
 G_\varepsilon(x)=0.049884.
\end{equation*}
Hence the failure of the isotropy condition
$BB^\top=\beta I_2$ alone is enough to destroy exact truncation.

\begin{figure}[htbp]
\centering

\begin{tikzpicture}
\begin{axis}[
    width=.67\textwidth,
    height=.52\textwidth,
    axis equal image,
    tick label style={font=\scriptsize},
    label style={font=\small},
     title style={font=\small},
    xlabel={$z_1$},
    ylabel={$z_2$},
    xmin=-.5,xmax=1.5,
    ymin=-.55,ymax=1.08,
    grid=major,
    title={Anisotropic control action: $G_\varepsilon=0.050$},
    legend style={font=\footnotesize,at={(0.03,0.03)},anchor=south west}
]
\addplot[densely dotted,domain=0:360,samples=120,no marks] ({0.5*cos(x)},{0.5*sin(x)});
\addlegendentry{$|z|=\varepsilon$}
\addplot[dashed,no marks] coordinates {
(1.000000,1.000000) (0.983193,0.983193) (0.966387,0.966387) (0.949580,0.949580) (0.932773,0.932773) (0.915966,0.915966) (0.899160,0.899160) (0.882353,0.882353) (0.865546,0.865546) (0.848739,0.848739) (0.831933,0.831933) (0.815126,0.815126) (0.798319,0.798319) (0.781513,0.781513) (0.764706,0.764706) (0.747899,0.747899) (0.731092,0.731092) (0.714286,0.714286) (0.697479,0.697479) (0.680672,0.680672) (0.663866,0.663866) (0.647059,0.647059) (0.630252,0.630252) (0.613445,0.613445) (0.596639,0.596639) (0.579832,0.579832) (0.563025,0.563025) (0.546218,0.546218) (0.529412,0.529412) (0.512605,0.512605) (0.495798,0.495798) (0.478992,0.478992) (0.462185,0.462185) (0.445378,0.445378) (0.428571,0.428571) (0.411765,0.411765) (0.394958,0.394958) (0.378151,0.378151) (0.361345,0.361345) (0.344538,0.344538) (0.327731,0.327731) (0.310924,0.310924) (0.294118,0.294118) (0.277311,0.277311) (0.260504,0.260504) (0.243697,0.243697) (0.226891,0.226891) (0.210084,0.210084) (0.193277,0.193277) (0.176471,0.176471) (0.159664,0.159664) (0.142857,0.142857) (0.126050,0.126050) (0.109244,0.109244) (0.092437,0.092437) (0.075630,0.075630) (0.058824,0.058824) (0.042017,0.042017) (0.025210,0.025210) (0.008403,0.008403)
};
\addlegendentry{Point-target optimal trajectory}
\addplot[thick,no marks] coordinates {
(1.000000,1.000000) (0.983634,0.983634) (0.967269,0.967269) (0.950903,0.950903) (0.934537,0.934537) (0.918171,0.918171) (0.901806,0.901806) (0.885440,0.885440) (0.869074,0.869074) (0.852708,0.852708) (0.836343,0.836343) (0.819977,0.819977) (0.803611,0.803611) (0.787245,0.787245) (0.770880,0.770880) (0.754514,0.754514) (0.738148,0.738148) (0.721782,0.721782) (0.705417,0.705417) (0.689051,0.689051) (0.672685,0.672685) (0.656320,0.656320) (0.639954,0.639954) (0.623588,0.623588) (0.607222,0.607222) (0.590857,0.590857) (0.574491,0.574491) (0.558125,0.558125) (0.541759,0.541759) (0.525394,0.525394) (0.509028,0.509028) (0.492662,0.492662) (0.476296,0.476296) (0.459931,0.459931) (0.443565,0.443565) (0.427199,0.427199) (0.410833,0.410833) (0.394468,0.394468) (0.378102,0.378102) (0.361736,0.361736)
};
\addlegendentry{Its truncation at first ball hit}
\addplot[very thick,no marks] coordinates {
(1.000000,1.000000) (0.989226,0.983424) (0.978451,0.966848) (0.967677,0.950272) (0.956903,0.933696) (0.946129,0.917120) (0.935354,0.900544) (0.924580,0.883968) (0.913806,0.867392) (0.903031,0.850816) (0.892257,0.834241) (0.881483,0.817665) (0.870708,0.801089) (0.859934,0.784513) (0.849160,0.767937) (0.838386,0.751361) (0.827611,0.734785) (0.816837,0.718209) (0.806063,0.701633) (0.795288,0.685057) (0.784514,0.668481) (0.773740,0.651905) (0.762966,0.635329) (0.752191,0.618753) (0.741417,0.602177) (0.730643,0.585601) (0.719868,0.569025) (0.709094,0.552449) (0.698320,0.535873) (0.687545,0.519298) (0.676771,0.502722) (0.665997,0.486146) (0.655223,0.469570) (0.644448,0.452994) (0.633674,0.436418) (0.622900,0.419842) (0.612125,0.403266) (0.601351,0.386690) (0.590577,0.370114) (0.579803,0.353538) (0.569028,0.336962) (0.558254,0.320386) (0.547480,0.303810) (0.536705,0.287234) (0.525931,0.270658) (0.515157,0.254082) (0.504382,0.237506) (0.493608,0.220930) (0.482834,0.204354) (0.472060,0.187779)
};
\addlegendentry{Ball-target optimal trajectory}
\addplot[only marks,mark=*] coordinates {(1,1)};
\node[anchor=west] at (axis cs:1,1) {$x=(1,1)$};
\end{axis}
\end{tikzpicture}

\caption{Failure caused by $BB^\top\ne\beta I_2$. The point-target optimal
trajectory and the ball-target optimal trajectory meet the same target ball
at different points and different times.}
\label{fig:num-B}
\end{figure}
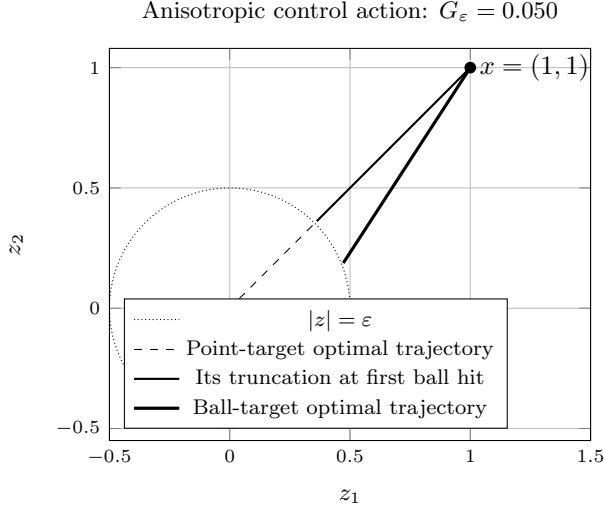

Next take
\begin{equation*}
 B=I_2,
 \qquad A=\operatorname{diag}(0,-2),
 \qquad x=(1,2)^\top,
 \qquad \varepsilon=1.
\end{equation*}
Now $BB^\top=I_2$, whereas
$A+A^\top=\operatorname{diag}(0,-4)$  is not a scalar multiple of $I_2$.
 The optimal
trajectories are computed by direct shooting of the Pontryagin system,
using the terminal normal to the target ball. Figure~\ref{fig:num-A} gives
\begin{equation*}
 T_\varepsilon^*(x)=0.399113,
 \qquad
 \widehat T_\varepsilon(x)=0.419552,
 \qquad
 G_\varepsilon(x)=0.020439.
\end{equation*}
Thus isotropy of the control action is not sufficient: anisotropy of the
symmetric part of the state dynamics also changes the optimal route to the
tolerance ball.

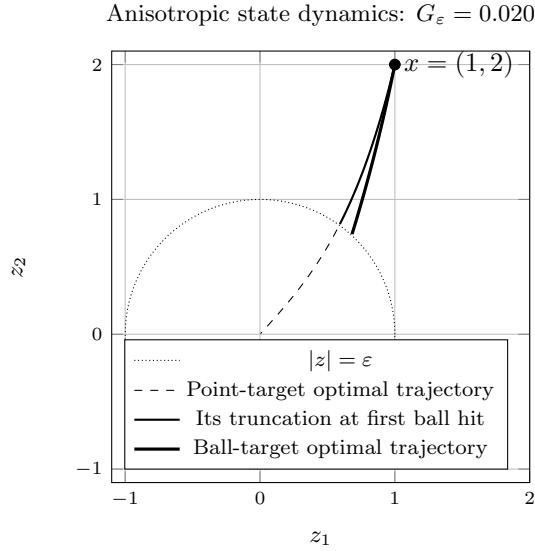
\begin{figure}[htbp]
\centering

\begin{tikzpicture}
\begin{axis}[
    width=.67\textwidth, 
    height=.56\textwidth, 
    axis equal image,
    tick label style={font=\scriptsize},
    label style={font=\small},
     title style={font=\small},
    xlabel={$z_1$},
    ylabel={$z_2$},
    xmin=-1.1,xmax=2,
    ymin=-1.1,ymax=2.1,
    grid=major,
    title={ Anisotropic state dynamics: $G_\varepsilon=0.020$},
    legend style={font=\footnotesize,at={(0.03,0.03)},anchor=south west}
]
\addplot[densely dotted,domain=0:360,samples=120,no marks] ({cos(x)},{sin(x)});
\addlegendentry{$|z|=\varepsilon$}
\addplot[dashed,no marks] coordinates {
(1.000000,2.000000) (0.987716,1.949631) (0.975437,1.900456) (0.963163,1.852447) (0.950894,1.805573) (0.938630,1.759805) (0.926372,1.715116) (0.914119,1.671478) (0.901873,1.628866) (0.889633,1.587251) (0.877399,1.546611) (0.865173,1.506918) (0.852954,1.468150) (0.840743,1.430282) (0.828540,1.393292) (0.816345,1.357157) (0.804159,1.321854) (0.791983,1.287363) (0.779816,1.253662) (0.767660,1.220732) (0.755514,1.188551) (0.743380,1.157100) (0.731258,1.126360) (0.719148,1.096313) (0.707050,1.066940) (0.694967,1.038224) (0.682898,1.010147) (0.670843,0.982692) (0.658804,0.955842) (0.646782,0.929582) (0.634776,0.903896) (0.622789,0.878767) (0.610820,0.854182) (0.598871,0.830126) (0.586942,0.806583) (0.575034,0.783540) (0.563149,0.760984) (0.551287,0.738901) (0.539449,0.717278) (0.527637,0.696101) (0.515851,0.675360) (0.504093,0.655041) (0.492363,0.635133) (0.480664,0.615625) (0.468997,0.596504) (0.457362,0.577761) (0.445761,0.559383) (0.434195,0.541362) (0.422667,0.523687) (0.411177,0.506348) (0.399726,0.489335) (0.388318,0.472639) (0.376953,0.456251) (0.365632,0.440162) (0.354358,0.424363) (0.343132,0.408846) (0.331957,0.393604) (0.320833,0.378627) (0.309763,0.363909) (0.298749,0.349442) (0.287792,0.335219) (0.276895,0.321233) (0.266059,0.307477) (0.255287,0.293945) (0.244581,0.280631) (0.233942,0.267528) (0.223373,0.254631) (0.212876,0.241934) (0.202453,0.229433) (0.192105,0.217121) (0.181836,0.204994) (0.171647,0.193046) (0.161540,0.181275) (0.151517,0.169675) (0.141581,0.158242) (0.131733,0.146973) (0.121976,0.135863) (0.112310,0.124909) (0.102739,0.114108) (0.093264,0.103457) (0.083887,0.092953) (0.074610,0.082592) (0.065434,0.072372) (0.056361,0.062291) (0.047393,0.052347) (0.038531,0.042537) (0.029776,0.032859) (0.021130,0.023311) (0.012594,0.013892) (0.004170,0.004599)
};
\addlegendentry{Point-target optimal trajectory}
\addplot[thick,no marks] coordinates {
(1.000000,2.000000) (0.991586,1.965370) (0.983174,1.931304) (0.974764,1.897794) (0.966357,1.864830) (0.957952,1.832401) (0.949549,1.800500) (0.941149,1.769116) (0.932751,1.738241) (0.924356,1.707866) (0.915964,1.677982) (0.907574,1.648581) (0.899188,1.619654) (0.890804,1.591193) (0.882424,1.563189) (0.874047,1.535635) (0.865673,1.508523) (0.857303,1.481845) (0.848936,1.455594) (0.840573,1.429761) (0.832213,1.404340) (0.823858,1.379323) (0.815507,1.354703) (0.807160,1.330473) (0.798817,1.306626) (0.790479,1.283156) (0.782145,1.260055) (0.773816,1.237317) (0.765492,1.214935) (0.757173,1.192904) (0.748860,1.171217) (0.740551,1.149867) (0.732249,1.128848) (0.723952,1.108156) (0.715661,1.087783) (0.707377,1.067724) (0.699098,1.047973) (0.690827,1.028525) (0.682562,1.009373) (0.674304,0.990514) (0.666053,0.971941) (0.657810,0.953649) (0.649574,0.935632) (0.641346,0.917887) (0.633127,0.900407) (0.624916,0.883189) (0.616713,0.866226) (0.608520,0.849514) (0.600336,0.833050) (0.592161,0.816827)
};
\addlegendentry{Its truncation at first ball hit}
\addplot[very thick,no marks] coordinates {
(1.000000,2.000000) (0.992759,1.964488) (0.985542,1.929499) (0.978347,1.895023) (0.971177,1.861051) (0.964030,1.827577) (0.956909,1.794591) (0.949814,1.762085) (0.942744,1.730052) (0.935702,1.698483) (0.928686,1.667372) (0.921698,1.636709) (0.914738,1.606489) (0.907808,1.576704) (0.900907,1.547346) (0.894036,1.518409) (0.887195,1.489886) (0.880386,1.461770) (0.873608,1.434054) (0.866863,1.406732) (0.860151,1.379798) (0.853472,1.353245) (0.846827,1.327067) (0.840217,1.301259) (0.833643,1.275813) (0.827104,1.250725) (0.820601,1.225989) (0.814136,1.201599) (0.807708,1.177549) (0.801318,1.153835) (0.794966,1.130451) (0.788654,1.107391) (0.782381,1.084652) (0.776149,1.062227) (0.769957,1.040112) (0.763806,1.018302) (0.757698,0.996792) (0.751631,0.975579) (0.745607,0.954657) (0.739626,0.934021) (0.733689,0.913669) (0.727796,0.893595) (0.721947,0.873795) (0.716143,0.854265) (0.710385,0.835001) (0.704672,0.816000) (0.699006,0.797257) (0.693385,0.778769) (0.687812,0.760532) (0.682285,0.742542)
};
\addlegendentry{Ball-target optimal trajectory}
\addplot[only marks,mark=*] coordinates {(1,2)};
\node[anchor=west] at (axis cs:1,2) {$x=(1,2)$};
\end{axis}
\end{tikzpicture}
\caption{Failure caused by $A+A^\top\not=2a I_2$, while $BB^\top=I_2$ remains
isotropic.}
\label{fig:num-A}
\end{figure}

\noindent{\bf 3. Approximate truncation near the rigid class.}
Theorem~\ref{wang122} shows that exact truncation of the optimal controls
is equivalent to exact truncation of the corresponding optimal
trajectories. This equivalence, however, is an exact one, and we do not
know whether an analogous equivalence holds at the approximate level. In
the final experiment, we therefore restrict attention to approximate
truncation at the trajectory level.

Starting from the rigid system in Figure~\ref{fig:num-rigid}, set
\begin{equation*}
 A_0=-\frac12I_2+\begin{pmatrix}0&-2\\2&0\end{pmatrix},
 \qquad B_0=I_2,
\end{equation*}
and perturb only the symmetric part of $A$ by
\begin{equation*}
 A_\delta=A_0+\delta\begin{pmatrix}1&0\\0&-1\end{pmatrix},
 \qquad B_\delta=B_0,
 \qquad 0\leq\delta\leq 0.2.
\end{equation*}
At $\delta=0$ the system satisfies the rigidity condition
\eqref{wang6}. For every $\delta>0$,
$A_\delta+A_\delta^\top$ is not a scalar multiple of $I_2$.
Hence, by Theorem~\ref{wang5}, universal exact truncation fails for the
perturbed system.

Fix $\varepsilon=0.2$ and consider three initial states of the same Euclidean
length,
\begin{equation*}
 x_1=(1,1)^\top,
 \qquad
 x_2=\sqrt2\begin{pmatrix}
 \cos(\pi/4+0.12)\\ \sin(\pi/4+0.12)
 \end{pmatrix},
 \qquad
 x_3=(-1,1)^\top.
\end{equation*}
Thus $x_1$ and $x_2$ are close, with $|x_1-x_2|\approx 0.170$, while
$|x_1-x_3|=2$ and $|x_2-x_3|\approx 1.876$. For each
$x\in\{x_1,x_2,x_3\}$, define
\begin{equation}\label{wang126}
 D_{\varepsilon,x}(\delta)
 :=\max_{0\leq t\leq T_{\varepsilon,\delta}^*(x)}
 \bigl|z_{\varepsilon,x}^{\delta,*}(t)-z_{0,x}^{\delta,*}(t)\bigr|,
\end{equation}
where the superscript $\delta$ indicates that the optimal trajectories
correspond to the system $(A_\delta,B_\delta)$. The optimal trajectories
of $(P_{0,x})$ and $(P_{\varepsilon,x})$ for the perturbed system are computed
by direct shooting of the Pontryagin system, with continuation in
$\delta$.

By Theorem~\ref{wang122}, $D_{\varepsilon,x}(0)=0$ is precisely the trajectory
representation of exact truncation for the unperturbed system.
Figure~\ref{fig:num-approx} shows that, for all three initial states,
$D_{\varepsilon,x_j}(\delta)$ remains small for small $\delta$ and tends
numerically to zero as $\delta\downarrow0$. The same behavior is observed
for the two nearby initial states $x_1,x_2$ and for the third state $x_3$,
which is far from both of them. Thus the numerical experiment provides
evidence for the trajectory-level approximate truncation phenomenon
suggested in Remark~\ref{wang127}. It is not intended as a proof of such
a stability result.

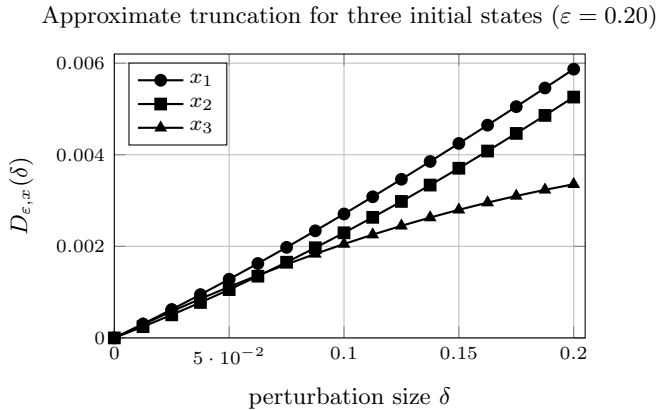
\begin{figure}[htbp]
\centering
\begin{tikzpicture}
\begin{axis}[
    width=.60\textwidth, 
    height=.41\textwidth, 
    tick label style={font=\small},
    label style={font=\small},
    tick label style={font=\scriptsize},
    label style={font=\small},
     title style={font=\small},
    xlabel={perturbation size $\delta$},
    ylabel={$D_{\varepsilon,x}(\delta)$},
    xmin=0,xmax=.205,
    ymin=0,ymax=.0062,
    xtick={0,0.05,0.10,0.15,0.20},
    grid=major,
    title={ Approximate truncation for three initial states ($\varepsilon=0.20$)},
    legend style={font=\footnotesize,at={(0.03,0.97)},anchor=north west},
    scaled y ticks=false,
    y tick label style={/pgf/number format/fixed,/pgf/number format/precision=3}
]
\addplot[thick,mark=*] coordinates {
(0.00000,0.0000000) (0.01250,0.0003056) (0.02500,0.0006212) (0.03750,0.0009466) (0.05000,0.0012815) (0.06250,0.0016253) (0.07500,0.0019778) (0.08750,0.0023384) (0.10000,0.0027068) (0.11250,0.0030824) (0.12500,0.0034647) (0.13750,0.0038532) (0.15000,0.0042474) (0.16250,0.0046467) (0.17500,0.0050505) (0.18750,0.0054583) (0.20000,0.0058694)
};
\addlegendentry{$x_1$}
\addplot[thick,mark=square*] coordinates {
(0.00000,0.0000000) (0.01250,0.0002458) (0.02500,0.0005037) (0.03750,0.0007735) (0.05000,0.0010552) (0.06250,0.0013484) (0.07500,0.0016531) (0.08750,0.0019689) (0.10000,0.0022956) (0.11250,0.0026330) (0.12500,0.0029806) (0.13750,0.0033381) (0.15000,0.0037052) (0.16250,0.0040813) (0.17500,0.0044662) (0.18750,0.0048592) (0.20000,0.0052600)
};
\addlegendentry{$x_2$}
\addplot[thick,mark=triangle*] coordinates {
(0.00000,0.0000000) (0.01250,0.0002952) (0.02500,0.0005797) (0.03750,0.0008532) (0.05000,0.0011157) (0.06250,0.0013668) (0.07500,0.0016065) (0.08750,0.0018345) (0.10000,0.0020508) (0.11250,0.0022554) (0.12500,0.0024481) (0.13750,0.0026289) (0.15000,0.0027979) (0.16250,0.0029550) (0.17500,0.0031003) (0.18750,0.0032339) (0.20000,0.0033558)
};
\addlegendentry{$x_3$}
\end{axis}
\end{tikzpicture}

\caption{Numerical evidence for trajectory-level approximate truncation.
The quantity $D_{\varepsilon,x}(\delta)$ in \eqref{wang126} is plotted for
$\varepsilon=0.2$ and the three initial states $x_1,x_2,x_3$. At $\delta=0$ exact
truncation gives zero discrepancy; the discrepancy remains small for small
anisotropic perturbations of the rigid system.}
\label{fig:num-approx}
\end{figure}

The first three figures visualize the exact rigidity theorem: the rigid
system exhibits exact truncation, while violation of either isotropy
identity changes the ball-target optimal route. The fourth figure has a
different role: it suggests that, although exact truncation is destroyed by
perturbations away from the rigid class, the corresponding optimal
trajectories may retain a quantitative form of approximate truncation.

\section*{Acknowledgments}
This work was supported by   the New Cornerstone
Science Foundation, the National Natural Science Foundation of China
under grants 12371450, 12671544, the Natural Science Foundation of Hebei Province, China under grant A2026202019
and the Shijiazhuang Science and Technology Bureau
 under grant 241791227A. The authors acknowledge the use of  ChatGPT for assistance in polishing  the English text and preparing the figures of this paper.






\end{document}

%% file: ex_shared.tex
\usepackage{lipsum}
\usepackage{amsfonts}
\usepackage{graphicx}
\usepackage{epstopdf}
\usepackage{algorithmic}
\usepackage{amssymb}
\usepackage{hyperref}
\ifpdf
  \DeclareGraphicsExtensions{.eps,.pdf,.png,.jpg}
\else
  \DeclareGraphicsExtensions{.eps}
\fi

\newsiamremark{remark}{Remark}
\newsiamremark{hypothesis}{Hypothesis}
\crefname{hypothesis}{Hypothesis}{Hypotheses}
\newsiamthm{claim}{Claim}
\newsiamremark{fact}{Fact}
\crefname{fact}{Fact}{Facts}

\headers{Exact Truncation and Radial Rigidity in Time-Optimal Control}{C. Quan, G. Wang, L. Wang, and Q. Yan}

\title{Exact Truncation and Radial Rigidity in Time-Optimal Control}

\author{
Changqin Quan\thanks{Graduate School of System Informatics, Kobe University, Kobe 657-8501, Japan; email:  quanchqin@gold.kobe-u.ac.jp. }
\and Gengsheng Wang
\thanks{Hetao Institute of Mathematics and Interdisciplinary Sciences, Shenzhen 518000, China;
e-mail:  wanggengsheng@himis-sz.cn. }
\and Lijuan Wang
\thanks{School of Mathematics and Statistics, Wuhan University, Wuhan 430072, China;
e-mail: ljwang.math@whu.edu.cn. }
\and Qishu Yan
\thanks{School of Science, Hebei University of Technology, Tianjin  300400, China;
e-mail: yanqishu@whu.edu.cn.}
}
\usepackage{amsopn}
